\documentclass{article}

     \PassOptionsToPackage{numbers, compress}{natbib}
     \usepackage[final]{neurips_2022}

\usepackage[utf8]{inputenc} % allow utf-8 input
\usepackage[T1]{fontenc}    % use 8-bit T1 fonts
\usepackage{hyperref}       % hyperlinks
\usepackage{url}            % simple URL typesetting
\usepackage{booktabs}       % professional-quality tables
\usepackage{amsfonts}       % blackboard math symbols
\usepackage{nicefrac}       % compact symbols for 1/2, etc.
\usepackage{microtype}      % microtypography
\usepackage{xcolor}         % colors

\usepackage{bibentry}
\usepackage[utf8]{inputenc} % allow utf-8 input
\usepackage[T1]{fontenc}    % use 8-bit T1 fonts

\usepackage{booktabs}       % professional-quality tables
\usepackage{amsfonts}       % blackboard math symbols
\usepackage{nicefrac}       % compact symbols for 1/2, etc.
\usepackage{microtype}      % microtypography
\usepackage{xcolor}
\usepackage[normalem]{ulem}
\definecolor{mydarkblue}{rgb}{0,0.08,0.45}
\usepackage{hyperref}
\hypersetup{
    unicode=false,          % non-Latin characters in AcrobatÕs bookmarks
    pdftoolbar=true,        % show AcrobatÕs toolbar?
    pdfmenubar=true,        % show AcrobatÕs menu?
    pdffitwindow=false,     % window fit to page when opened
    pdfstartview={FitH},    % fits the width of the page to the window
    pdfnewwindow=true,      % links in new window
    colorlinks=true,       % false: boxed links; true: colored links
    linkcolor=mydarkblue,          % color of internal links (change box color with linkbordercolor)
    citecolor=mydarkblue,        % color of links to bibliography
    filecolor=magenta,      % color of file links
    urlcolor=cyan           % color of external links
}

\usepackage{algorithm}
\usepackage{amsmath,amssymb,amsthm,amsfonts,latexsym}        
\usepackage{mathtools}
\usepackage{graphicx}
\usepackage{bm}
\usepackage{pifont}
\usepackage{color}
\usepackage{appendix}
\usepackage{enumitem}
\usepackage{environ}         % provides \BODY
\usepackage{etoolbox}        % provides \ifdimcomp
\usepackage{setspace}
\usepackage[capitalise,nameinlink]{cleveref}
\usepackage{dsfont}
\usepackage{thmtools} 
\usepackage{thm-restate}
\usepackage{lipsum}
\usepackage{float}
\usepackage[algo2e]{algorithm2e} 
\usepackage{thm-restate}
\usepackage{mdframed}
\usepackage{titlesec}
\usepackage{xpatch}
\usepackage{placeins} 
\usepackage{wrapfig}
\usepackage{caption}
\usepackage{subcaption}
\usepackage[normalem]{ulem}

\newcommand{\nleft}{\mathclose\bgroup\left}
\newcommand{\nright}{\aftergroup\egroup\right}

\newcommand{\Lg}{\beta}
\newcommand{\muw}{\nu}
\newcommand{\muh}{\mu}
\newcommand{\G}{G}
\newcommand{\h}{H}
\newcommand{\w}{w}
\newcommand{\Eset}{\mathbf{V}}
\newcommand{\N}{k}
\newcommand{\D}{d}
\newcommand{\support}{s}
\newcommand{\varsq}{\sigma^2}
\newcommand{\xstar}{x_\star}
\newcommand{\ystar}{y_\star}

\newcommand{\tauk}{\tau_{k}}
\newcommand{\inner}[2]{\langle #1,\, #2 \rangle}

\newcommand{\norm}[1]{\left\|#1\right\|}

\newcommand{\normsq}[1]{\left\|#1\right\|^{2}}

\newcommand{\E}{\mathbb{E}}

\newcommand{\fkopt}{f_{i_k}\kern-.65em{}^{\raisebox{1pt}{\tiny $*$}}\kern.05em}

\DeclareMathOperator*{\argmax}{arg\,max}
\DeclareMathOperator*{\argmin}{arg\,min}

\newtheorem{theorem}{Theorem}
\newtheorem{lemma}{Lemma}
\newtheorem{proposition}{Proposition}
\newtheorem{assumption}{Assumption}
\newtheorem{remark}{Remark}
\newtheorem{definition}{Definition}

\usepackage{amsmath}
\usepackage{amssymb}
\usepackage{amsthm}
\usepackage{mathrsfs}
\usepackage{amsfonts}
\usepackage{bm}
\usepackage{xparse}

\def\1{\bm{1}}

\DeclareMathAlphabet{\mathsfit}{\encodingdefault}{\sfdefault}{m}{sl}
\SetMathAlphabet{\mathsfit}{bold}{\encodingdefault}{\sfdefault}{bx}{n}

\newcommand{\R}{\mathbb{R}}

\usepackage{thmtools, thm-restate}

\newcommand{\tightsub}[1]{{\kern -.1em \raise-.1em\hbox{\tiny$#1$}}{}}
\DeclareDocumentCommand{\Normal}{g}{\mathop{\mathcal{N}}\IfNoValueF{#1}{\nleft(#1\nright)}}
\DeclareDocumentCommand{\bigO}{g}{\mathop{\mathcal{O}}\IfNoValueF{#1}{\nleft(#1\nright)}}

\NewDocumentCommand{\paren}{sm}{%
  \IfBooleanTF{#1}{(#2)}{\nleft(#2\nright)}
}
\NewDocumentCommand{\braces}{sm}{%
  \IfBooleanTF{#1}{(#2)}{\nleft\{#2\nright\}}
}

\newlength{\dittowidth}

\mdfdefinestyle{mdframedthmbox}{%
	leftmargin=.0\textwidth,
	rightmargin=.0\textwidth,%
	innertopmargin=0.75em,
	innerleftmargin=.5em,
	innerrightmargin=.5em,
}

\usepackage{pgfplots}
\usetikzlibrary{external}
\pgfplotsset{
    compat=1.13
    }
\tikzstyle{loosely dashed}=          [dash pattern=on 2pt off 4pt]
\pgfplotsset{plotOptions1/.style={%
		height=.34\linewidth,
		width=.45\linewidth,
		ymax=30000,ymin=.0001,
        xmin=1,xmax=100000,
		label style={font=\footnotesize},
		legend style={font=\footnotesize},
		ylabel={},
		xlabel={Iteration},
		ytick={.0001,.001,.01,.1,1, 10, 100, 1000, 10000, 100000},
		xtick={1, 10, 100, 1000, 10000, 100000},
        yticklabels={,,$10^{-3}$,,$10^{-1}$, , $10^1$, , $10^3$, , $10^5$},
        xticklabels={, $10^1$, , $10^3$, , $10^5$},
		tick label style={font=\scriptsize},
		grid=major,
		grid style={gray!25},
        axis x line*=none,
        axis y line*=none,
        x axis line style={gray!50},
        xtick style={gray!50},
        y axis line style={gray!50},
        ytick style={gray!50},
        xlabel near ticks
	}}
\pgfplotsset{plotOptions3/.style={%
		width=.3\linewidth,
		ymax=10,ymin=.000001,
		xmin=1,xmax=100000,
		label style={font=\footnotesize},
		legend style={font=\footnotesize},
		ytick={.00001,.0001,.001,.01,.1,1, 10},
		xtick={1, 10, 100, 1000, 10000, 100000},
        xticklabels={, $10^1$, , $10^3$, , $10^5$},
        yticklabels={$10^{-5}$,,$10^{-2}$,,$10^{-1}$,, $10^1$},
		ylabel={},
		xlabel={Iteration},
		grid=major,
		grid style={gray!25},
		tick label style={font=\scriptsize},
        axis x line*=none,
        axis y line*=none,
        x axis line style={gray!50},
        xtick style={gray!50},
        y axis line style={gray!50},
        ytick style={gray!50},
        xlabel near ticks
	}}

\title{Fast Stochastic Composite Minimization and an Accelerated Frank-Wolfe Algorithm under Parallelization}
\author{%
  Benjamin Dubois-Taine\\
  DI ENS,  Ecole normale supérieure,\\ Université PSL, CNRS, INRIA\\
  75005 Paris, France\\
  \texttt{benjamin.paul-dubois-taine@inria.fr} 
  \And
  Francis Bach\\
  DI ENS,  Ecole normale supérieure,\\ Université PSL, CNRS, INRIA\\
  75005 Paris, France\\
  \texttt{francis.bach@inria.fr} 
  \And
    Quentin Berthet\\
  Google Research, Brain team, Paris\\
  \texttt{qberthet@google.com}
   \And
  Adrien Taylor\\
  DI ENS,  Ecole normale supérieure,\\ Université PSL, CNRS, INRIA\\
  75005 Paris, France\\
  \texttt{adrien.taylor@inria.fr} 
}

\usepackage[colorinlistoftodos]{todonotes}
\begin{document}

\maketitle
\begin{center}
    {\centering\color{red} (Update November 21, 2024: a minor error at the end of Appendix~\ref{sec:composite_proof} impacts a few constants throughout the paper. The correction is propagated in red in the whole document)}
\end{center}
\begin{abstract}

We consider the problem of minimizing the sum of two convex functions. One of those functions has Lipschitz-continuous gradients, and can be accessed via stochastic oracles, whereas the other is ``simple''. We provide a Bregman-type algorithm with accelerated convergence in function values to a ball containing the minimum. The radius of this ball depends on problem-dependent constants, including the variance of the stochastic oracle. We further show that this algorithmic setup naturally leads to a variant of Frank-Wolfe achieving acceleration under parallelization. More precisely, when minimizing a smooth convex function on a bounded domain, we show that one can achieve an $\epsilon$ primal-dual gap (in expectation) in $\tilde{O}(1 /\sqrt{\epsilon})$ iterations, by only accessing gradients of the original function and a linear maximization oracle with $O(1 / \sqrt{\epsilon})$ computing units in parallel. We illustrate this fast convergence on synthetic numerical experiments.
\end{abstract}

\section{Introduction}
\label{sec:introduction}
We consider the composite minimization problem
\begin{align}
    \label{eq:composite-min}
    \min_{y\in \Eset
    } \ \left\{F(y) := \G(y) + \h(y)\right\},
\end{align}
where $\Eset$ is some real vector space equipped with a norm $\norm{\cdot}$ and an inner product $\langle\cdot,\cdot\rangle$, see Section~\ref{sec:notations} for details. For instance, one can consider $\Eset= \R^d$ or $ \Eset = \R^{p\times q}$ equipped with the standard Euclidean inner product and norm. The function $\G$ is convex and $\Lg$-smooth and the function $\h$ is $\muh$-strongly convex, both with respect to (w.r.t.) some norm $\|.\|$ on $\Eset$ (see~\cref{sec:notations} for precise definitions). We assume that $\h$ is ``simple'', meaning that one can efficiently solve
\begin{align}\label{eq:bregman_oracle}
    \argmin_{y \in \Eset} \  \big\{\inner{z}{y} + \h(y) + \alpha D_\w(y, z_0)\big\},
\end{align}
where $D_\w$ is the Bregman divergence induced by a strongly convex function~$\w$. For instance, when $\w$ is the standard squared Euclidean norm, this amounts to computing the proximal operator of~$\h$. When deterministic gradients of $\G$ are available, an accelerated method relying on~\eqref{eq:bregman_oracle} and achieving rates of the form $O\left(\exp\left( - k \sqrt{\frac{\muh}{\Lg}}\right) \right)$ was proposed by~\citet[AGD+]{diakonikolas2021complementary}.

In this work, we are interested in the case where $\G$ is only available through a stochastic oracle. In particular, we provide an accelerated algorithm converging in function values to a neighborhood of the minimum with the same rate as above. The size of the neighborhood is of order $O\left(\frac{\sigma^2}{\sqrt{\muh \Lg}}\right)$ where $\sigma^2$ is the variance of the stochastic gradients. The dependence of the noise term on $1 / \sqrt{\muh \Lg}$ is similar to that of previous stochastic accelerated methods in simpler settings, see, e.g.,~\citep{aybat2019universally}. Although our rate is an extension of the result in~\citep{diakonikolas2021complementary}, the parameters are different and tailored for the stochastic setup.

In~\cref{sec:AFW}, we focus on minimization problems over a compact convex set $K$ for which we have access to a linear optimization oracle, just as in the Frank-Wolfe algorithm. Formally, we consider solving $\min_{x\in K} f(x)$ where $f$ is convex with Lipschitz-continuous gradients. In that case, the dual problem can be seen as a particular instance of~\eqref{eq:composite-min} where $H:=f^*$ is the Fenchel conjugate of $f$, and $G$ is a smoothed version of the support function of $K$. Such smoothed functions can only be accessed through a stochastic oracle whose computation boils down to solving a linear optimization problem on~$K$. An appropriate choice of~$\w$ allows us to use the algorithmic framework by only computing  gradients of~$f$ and solving linear optimization subproblems over~$K$. In short, minimizing $f$ over the set $K$ with an~$\epsilon$ primal-dual gap requires $\tilde{O}\left( \max\left\{\frac{1}{\sqrt{\epsilon}}, \frac{1}{\epsilon m}\right\}\right)$ iterations when $m$ computing units can be used in parallel (each of them only solving linear optimization subproblems on~$K$). In particular, when $m = \nicefrac{1}{\sqrt{\epsilon}}$, this effectively yields an accelerated rate of $\tilde{O}\left( \nicefrac{1}{\sqrt{\epsilon}}\right)$ iterations with a single gradient evaluation per iteration to achieve the required accuracy.

We emphasize that the first contribution of this work is to provide an accelerated method for general stochastic composite problems (in~\cref{sec:stochastic-composite-minimization}). In~\cref{sec:AFW}, we then show that it can be used for obtaining an accelerated Frank-Wolfe algorithm.

\subsection{Related Work}

As previously emphasized, the main algorithmic ingredients which inspired this work were developped by~\citep{diakonikolas2021complementary} and~\citep{gasnikov2018universal}. For further references on acceleration and gradient-based composite convex minimization, we refer to the original works~\citep{nesterov1983method,nesterov2013gradient} or to a recent survey~\citep{d2021acceleration}. For the stochastic setup, we refer to~\citep{lan2012optimal}. We refer to~\citep{nesterov2003introductory,devolder2013exactness} for further references and pointers on those topics.

Let us briefly describe the main differences between this work and~\citep{diakonikolas2021complementary,gasnikov2018universal}. First,~\citet{gasnikov2018universal} obtain accelerated rates similar to ours when the strong convexity assumption is placed on~$\G$ instead of $\h$ here (which has practical consequences, including well-posedness of~\eqref{eq:bregman_oracle} and the size of the constants $\Lg$ and $\muh$, see discussion in~\citep[Introduction]{diakonikolas2021complementary}). When the gradients of $\G$ are stochastic,~\cite{gasnikov2018universal} provides convergence rates, but only when the underlying norm is the Euclidean norm. Finally, whereas~\citet{diakonikolas2021complementary} consider more general assumptions in the deterministic non-Euclidean case, our work also covers stochastic approximations. We emphasize that each assumption we make (namely accelerated, stochastic, proximal Bregman methods) is necessary to yield acceleration of Frank-Wolfe under parallelization in~\cref{sec:AFW}.

Two key observations underlying this work are on the one hand the duality link between Frank-Wolfe methods and Bregman methods~\cite{bach2015duality}, and on the other hand randomized smoothing techniques~\cite{nesterov2017random, duchi2012randomized, abernethy2014online, abernethy2016perturbation, berthet2020learning, jaxopt_implicit_diff} which can be naturally computed using linear optimization steps. The main question of interest was whether accelerated Bregman methods should naturally give rise to accelerated Frank-Wolfe methods on the dual.

The Frank-Wolfe (FW) method (a.k.a.~conditional gradient method) and its variants were first introduced by~\citet{frank1956algorithm} (see also~\cite{jaggi2013revisiting, jaggi2011sparse} for  more modern presentations). When considering optimization over a convex set $K$,  classical first-order methods are often naturally embedded with a projection operator (onto $K$). Depending on $K$, those projections are potentially costly. An alternate approach for taking constraints into account within first-order method consists in using linear optimization oracles (a.k.a.~Frank-Wolfe techniques). In many applications, such linear minimizations are much cheaper than projecting onto the feasible set, see for example~\cite{freund2017extended, harchaoui2015conditional}. Despite its wide use in practical applications, the main drawback of the Frank-Wolfe method lies in its slow convergence rate of $O( \nicefrac{1}{\epsilon})$, standing in sharp contrast with the $O( \nicefrac{1}{\sqrt{\epsilon}})$ convergence of Nesterov's accelerated gradient descent~\cite{nesterov1983method} relying on projections.

As ``purely accelerated'' rates of convergence are out of reach for vanilla Frank-Wolfe methods (see lower complexity bound in~\cite{lan2013complexity}), most works on accelerating Frank-Wolfe have focused on exploiting specific additional assumptions on the problems at hand. Common such assumptions include strong convexity of the feasible set \cite{demianov1970approximate, levitin1966constrained, garber2015faster}, and strong convexity of the objective function along with the assumption that the minimizer lies in the interior of the feasible set~\cite{guelat1986some,lacoste2015global}. In both cases, Frank-Wolfe is known to improve on the $O(\nicefrac{1}{\epsilon})$ rate. Some efforts have also gone into finding rates matching performances of Nesterov's accelerated gradient descent \cite{nesterov1983method} without strong convexity. In particular, when $K$ is a polytope and when a certain type of constraint qualification is satisfied, Frank-Wolfe converges asymptotically at rate $O(\nicefrac{1}{k^2})$~\cite{bach2020effectiveness}. Adding momentum to Frank-Wolfe has also been studied, and accelerated rates are attained on some $\ell_p$-norm balls when the minimizer is on the boundary of the feasible set~\cite{li2020does}. Our approach in this work is orthogonal to the aforementioned results, in that we do not make additional assumptions on the objective function or the feasible set, but instead show that parallelization can help reaching accelerated rates for a variant of Frank-Wolfe. Note that \cite{lan2016conditional} manages to reach a similar rate of $O(\nicefrac{1}{\sqrt{\epsilon}})$ iterations for a variant of Frank-Wolfe, where each iteration requires one gradient evaluation and $O(\nicefrac{1}{\sqrt{\epsilon}})$ calls to a linear optimization oracle. However, in contrast with our algorithm, their approach is non-parallelizable as the calls to the linear optimization oracle need to be made sequentially.

The rest of the paper is organized as follows. In~\cref{sec:notations} we define notations and review some classical definitions from convex analysis. In~\cref{sec:stochastic-composite-minimization} we provide worst-case rates for a stochastic Bregman dual-averaging-type algorithm, along with some intuitions and proof sketches. In~\cref{sec:AFW} we show how a Frank-Wolfe method directly fits within this framework, thereby obtaining accelerated worst-case guarantees on the primal-dual gap. In~\cref{sec:experiments}, we illustrate our theoretical findings on a set of simple numerical examples.
\section{Notation and definitions}
\label{sec:notations}

Formally, we consider a real finite-dimensional vector space $\Eset$ and its dual space $\Eset^*$ consisting of all linear functions on $\Eset$, as well as a dual pairing denoted by $\langle\cdot,\cdot\rangle: \Eset^*\times\Eset\rightarrow \mathbb{R}$, and a norm $\norm{\cdot}:\Eset\rightarrow \mathbb{R}$. We also consider the corresponding dual norm $\norm{\cdot}_*:\Eset^*\rightarrow \mathbb{R}$ induced by the choice of $\norm{\cdot}$ and $\langle\cdot,\cdot\rangle$ using $\norm{z}_* = \sup_{\norm{x} \leq 1} \inner{z}{x}$. For instance, one can use $\Eset =\Eset^*= \R^d$ or $ \Eset =\Eset^*= \R^{p\times q}$ equipped with the standard Euclidean inner product and norm. We insist on the fact that the formal notation $\Eset$ and $\Eset^*$ is used mostly for emphasizing differences between primal and dual spaces/problems below.
For a closed, proper, convex function $\Psi : \Eset \rightarrow \R$, we denote $\partial \Psi(x)$ the set of all subgradients of $\Psi$ at $x$. When $\partial \Psi(x)$ is a singleton, we denote its single element by $\nabla \Psi(x)$.

\begin{definition} 
$\Psi : \Eset \rightarrow \R$ is $L$-smooth w.r.t. $\norm{\cdot}$ if for all $x, y \in \Eset$,
\begin{align}
    \norm{\nabla \Psi(x) - \nabla \Psi(y)}_* \leq L \norm{x - y}.
\end{align}
\end{definition}
The following proposition will be helpful throughout the paper (see \cite{nesterov2003introductory} for a proof).
\begin{proposition}
$\Psi : \Eset \rightarrow \R$ is $L$-smooth w.r.t. $\norm{\cdot}$ if and only if for all $x, y \in \Eset$,
\begin{align*}
    \Psi(x) \leq \Psi(y) + \inner{\nabla \Psi(y)}{x - y} + \frac{L}{2} \normsq{x - y}.
\end{align*}
\end{proposition}
\begin{definition}
$\Psi : \Eset \rightarrow \R$ is $\mu$-strongly convex w.r.t. $\norm{\cdot}$ if for all $x, y \in \Eset$ and any $g_\Psi(y) \in \partial \Psi(y)$,
\begin{align}
    \Psi(x) \geq \Psi(y) + \inner{g_\Psi(y)}{x - y} + \frac{\mu}{2} \normsq{x - y}.
\end{align}
\end{definition}

\begin{definition}
For a closed, proper, convex function $\Psi : \Eset \rightarrow \R$, its conjugate function is defined as
\begin{align}
    \Psi^*(z) = \sup_{x \in \Eset} \ \  \inner{z}{x} - \Psi(x).
\end{align}
\end{definition}
For later references, we need a few results related to smoothing the support function of a set $K\subset\Eset^*$ ($K$ will appear as a convex set in the dual problem to~\eqref{eq:composite-min}).
\begin{definition}
For a set $K \subset \Eset^*$, we define $I_K$ as the indicator function of $K$, i.e., $I_K(x) = 0$ for $x\in K$ and $I_K(x) = \infty $ for $x \in \Eset^* \setminus K$. The support function of $K$ is defined as
\begin{align}
    \support(y) = \sup_{x \in K} \  \inner{x}{y}.
\end{align}
\end{definition}
If $K$ is non-empty, convex and closed, the indicator and support functions of $K$ are conjugates of each other, i.e., $(I_K)^* = \support$ and $\support^* = (I_K)^{**} = I_K$~\cite{rockafellar2015convex}. The support function is convex yet not differentiable in general applications. For the purpose of our work, we consider smoothing the support function via randomization. Such stochastic smoothing has recently gained popularity in the optimization literature, see for example~\cite{nesterov2017random, duchi2012randomized, abernethy2014online, abernethy2016perturbation, berthet2020learning}. We define the main tools and state the relevant properties behind this idea.
\begin{definition}
\label{def:smoothed-support}
For a set $K\subset \Eset^*$, a scalar $\alpha > 0$ and a random variable $\Delta$, we define the smoothed support function of $K$ as
\begin{align}
    \support_\alpha(y) = \E_\Delta \left[ \support(y + \alpha \Delta)\right] = \E_\Delta \left[ \sup_{x \in K}\  \inner{x}{y+ \alpha \Delta} \right].
\end{align}
\end{definition}
We will use the following proposition on the smoothed support function, proved in~\cite{berthet2020learning}.
\begin{proposition}
\label{prop:perturbed-opt}
Suppose $K\subset \Eset^*$ is convex compact and let $R_K = \max_{x \in K} \norm{x}_*$. Suppose the random variable $\Delta$ has positive differentiable density $d\pi(z) \propto \exp(- \eta(z))dz$ for some function~$\eta(\cdot)$, and let $M^2 = \E\left[ \normsq{\nabla_z \eta(Z)}_*\right]$. Then $\support_\alpha$ is convex, $\frac{R_K M}{\alpha}$-smooth w.r.t. $\norm{\cdot}$ and for any $y \in \Eset$,
\begin{align*}
        &\nabla \support_\alpha(y) = \E_\Delta \left[ \argmax_{x \in K} \inner{x}{y + \alpha \Delta}\right]
        \quad \text{ and }\quad 
        \support(y) \leq \support_\alpha(y) \leq \support(y) + \alpha \support_1(0).
    \end{align*}
\end{proposition}
\section{Fast Stochastic Composite Minimization}
\label{sec:stochastic-composite-minimization}

In this section we focus on solving the composite problem
\begin{align}
    \label{prob:composite}
    \min_{y \in \Eset} \ \big\{ F(y) := \G(y) + \h(y)\big\},
\end{align}
where $\G$ is convex and $\Lg$-smooth, and $\h$ is $\muh$-strongly convex w.r.t. $\norm{\cdot}$. \cref{alg:acc-stoch-bregman} summarizes our proposed algorithm, which resembles the general AGD+ algorithm from~\cite{diakonikolas2021complementary}. Notable differences include access to gradients of $\G$ through a stochastic oracle, and explicit dependencies on the smoothness and strong convexity constants for the different updates. Moreover, the presented analysis is different and specifically tailored to handle stochasticity in the gradients of $\G$.

An actual implementation of~\cref{alg:acc-stoch-bregman} requires the ability to efficiently solve the minimization step~\eqref{alg1:minimization-step}, which should be well-defined. This intermediate minimization subproblem is often referred to as a Bregman proximal problem, and a sufficient condition for this operation to be well-defined is to require $w$ to be strongly convex. In the case where $\w$ is the Euclidean norm,~\eqref{alg1:minimization-step} amounts to computing the proximal operator of $\h$. Considering a general regularizer $\w$ has several benefits, in that the Euclidean norm might not well capture the geometry of the problem, and because particular choices of $\w$ might make~\eqref{alg1:minimization-step} easily solvable. The latter will become particularly clear in~\cref{sec:AFW}. 

Before we move on to the analysis, let us emphasize that the first step of~\cref{alg:acc-stoch-bregman} consisting in picking $z_0 \in \argmin_y \w(y)$ is not restrictive. Indeed one could instead pick any $z_0 \in \Eset$ and set $\tilde{w}(y) = w(y) - \inner{g_{w}(z_0)}{y}$ where $g_w(z_0) \in \partial w(z_0)$. One could then run~\cref{alg:acc-stoch-bregman} with the ``shifted'' $\tilde{w}$ instead of $w$. Doing so does not change the complexity of the minimization step~\eqref{alg1:minimization-step}, and the following analysis remains unchanged.

We are now ready to analyse~\cref{alg:acc-stoch-bregman}. For this purpose let us define
\begin{align*}
    c_k := \sum_{i=0}^{k-1} (A_{i+1} - A_i) (\G(y_i) - \inner{\nabla \G(y_i)}{y_i}).
\end{align*}
where the sequences $\{A_k\}_{k\in \mathbb{N}}$ and $\{y_k\}_{k\in \mathbb{N}}$ are defined in~\cref{alg:acc-stoch-bregman}. We also use the notation $\E_k$ for denoting the expectation at iteration $k$ conditioned on the previous iterations (that is, $\E_k$ shortens $\E_k[\,\cdot\,]=\E[\,\cdot\, | \,y_k,\,z_k,\,d_k,\,c_k]$), while $\E$ denotes the total expectation.  Before we can state our results, we need one more assumption on the stochastic gradients.
\begin{assumption}
\label{assumption:bounded-variance}
For any $k \in \mathbb{N}$, the stochastic gradients satisfy $\E_k[g_k] = \nabla G(v_k)$ and $\E_k\normsq{g_k - \nabla \G(v_k)}_* \leq \varsq$.
\end{assumption}

\begin{algorithm}[ht]
\caption{Stochastic Composite Minimization}
\label{alg:acc-stoch-bregman}
\textbf{Input}: ($\Lg, \muh, \muw)$. $\Lg$-smooth function $\G$,  $\muh$-strongly convex function $\h$, $\muw$-strongly convex function $\w$, all w.r.t. the same norm $\norm{\cdot}$.\\
Pick $z_0 \in \argmin_y \w(y)$ and set $y_0 = z_0$.\\
Set $A_0 = 0$ and $d_0 = 0$.\\
\For{$k =0, 1, \dots$}{
    \begin{align}
    \label{alg1:quadratic_update}
    A_{k+1} &= \tfrac{A_k(\muh + 2\Lg + \sqrt{\muh \Lg}) + \Lg \muw + \sqrt{(\Lg \muw +  \muh A_k)^2 + 4A_k\Lg^2 \muw + 5 A_k^2 \muh \Lg + 2A_k \sqrt{\muh \Lg}( \Lg \muw + A_k \muh) }}{2(\Lg + \sqrt{\muh \Lg})}
     \\
    \tau_k&=1-\tfrac{A_k}{A_{k+1}}\\
    v_k &= (1 - \tauk)y_k + \tauk z_k\\
    \intertext{Compute a stochastic gradient $g_k$ of function $G$ at iterate $v_k$,}
    d_{k+1} &= d_k + (A_{k+1} - A_k) g_k\\
    \label{alg1:minimization-step}
    z_{k+1} &\in \argmin_{y \in \Eset} \left\{ \inner{d_{k+1}}{y} + A_{k+1}\h(y) + \Lg \w(y)\right\}\\
    y_{k+1} &= (1 - \tauk)y_k + \tauk z_{k+1}
    \end{align}
}
\end{algorithm}

The assumption on the variance of the stochastic gradients is common when studying stochastic first-order methods, and allows proving the next proposition, which relates consecutive iterations.
\begin{proposition}\label{prop:Lyap-breg}
Suppose~\cref{assumption:bounded-variance} holds and let $m_k(y) := \inner{d_k}{y} + c_k + A_{k}\h(y) + \Lg \w(y)$. At iteration $k$, the iterates of~\cref{alg:acc-stoch-bregman} satisfy
\begin{align}
    \label{eq:Lyap-recursion}
    \E_k\left[ A_{k+1} F(y_{k+1}) - m_{k+1}(z_{k+1}) \right] \leq A_k F(y_k) - m_k(z_k) + (A_{k+1} - A_k)\frac{\varsq}{2 \sqrt{\muh \Lg}}.
\end{align}
\end{proposition}
We defer the full proof to~\cref{app:composite-proofs}, but highlight the main steps here. This way we also hope to shed light on the update of $A_{k+1}$. The first step is to compute an upper bound on $A_{k+1}\G(y_{k+1})$ depending on $A_k G(y_k)$. The second step is similar, and computes an upper bound on $A_{k+1}\h(y_{k+1})$ depending on $A_k \h(y_k)$. Summing up the two inequalities and taking expectations yields inequality~\eqref{eq:Lyap-recursion} with an additional term in $\normsq{z_{k+1} - z_k}$. To exactly obtain inequality~\eqref{eq:Lyap-recursion} we set $A_{k+1}$ so as to cancel out the coefficient multiplying $\normsq{z_{k+1} - z_k}$. This turns out to be equivalent to setting $A_{k+1}$ as the root of a quadratic polynomial, explaining the form of update~\eqref{alg1:quadratic_update} in~\cref{alg:acc-stoch-bregman}.

\cref{prop:Lyap-breg} allows us to get the final rate of convergence of~\cref{alg:acc-stoch-bregman}. We again defer the proof to~\cref{app:composite-proofs}. The idea here is first to unroll the recursion in~\eqref{eq:Lyap-recursion} so as to get a constant term on the right-hand side of the inequality. We then relate $m_N(z_N)$ to $m_N(\ystar)$ and to the minimum of $F$ and finally, we show that $A_{k+1} \geq A_k \left( 1+ \frac{\sqrt{\muh}}{2 \left( \sqrt{\Lg} + \sqrt{\muh} \right)}\right)$, which gives the exponential decay term.
\begin{theorem}
\label{thm:composite-result}
Suppose~\cref{assumption:bounded-variance} holds, let $\ystar \in \argmin F$ and define $D_w(\ystar, y_0) = w(\ystar) - w(y_0) \geq 0$. The convergence rate of~\cref{alg:acc-stoch-bregman} after $\N$ iterations is
\begin{align}
    \E[F(y_\N) - F(\ystar)] \leq {\color{red}\frac{\sqrt{\beta} +\sqrt{  \mu }}{\nu \sqrt{\beta}  }}\exp\left( - \frac{{\color{red}(\N-1)} \sqrt{\muh}}{{\color{red} 4} \left( \sqrt{\Lg} + \sqrt{\muh}\right)} \right)\Lg D_\w(\ystar, y_0) + \frac{\varsq}{2\sqrt{\muh \Lg}}.
\end{align}
\end{theorem}
This rate is a typical accelerated rate. In the (usual) case where $\muh \leq \Lg$, the exponential decay is bounded above by $\exp\left(-\frac{{\color{red}(\N - 1)}}{{\color{red} 8}} \sqrt{\frac{\muh}{\Lg}}\right)$, which shows the natural dependence on $\sqrt{\frac{\muh}{\Lg}}$. In addition, the neighborhood term is of the form $O\left( \frac{\varsq}{\sqrt{\muh \Lg}}\right)$, which is again typical of accelerated stochastic methods~\cite{aybat2019universally}. Note that $ w(\ystar) - w(y_0)$ is equal to the Bregman divergence $D_w(\ystar, y_0):=w(\ystar) - w(y_0)-\langle g_w(y_0); \ystar-y_0\rangle$ (with $g_w(y_0)\in\partial w(y_0)$) through the choice $g_w(y_0)=0\in\partial w(y_0)$, which is valid as $y_0$ minimizes $w(\cdot)$. We emphasize that we do not require differentiability of the function $\w$ anywhere.

In the next section, we show how the above algorithm can be directly applied to a smoothed dual of a minimization problem over a compact convex set, yielding a variant of the   Frank-Wolfe algorithm which can achieve accelerated rates under parallelization.

\section{Accelerating Frank-Wolfe with parallelization} 
\label{sec:AFW}

We consider the following minimization problem over a compact convex set $K \subset \Eset^*$
\begin{align}
    \min_{x \in \Eset^*}\ \big\{  f(x) + I_K(x) \big\},
\end{align}
where $f$ is convex and $L$-smooth w.r.t. $\norm{\cdot}_*$. Its Fenchel-Rockafellar dual \cite{rockafellar2015convex} reads
\begin{align}
    \max_{y \in \Eset}\hspace{0.5ex} \big\{\D(y) :=   -\support(-y) - f^*(y) \big\} .
\end{align}
The smoothness of $f$ implies that $f^*$ is $\nicefrac{1}{L}$-strongly convex w.r.t. $\norm{\cdot}$ \cite[Proposition 12.60]{rockafellar2009variational}. The term in $\support$ is however not smooth w.r.t. $\norm{\cdot}$,  which prevents us from directly applying results from the previous section. Instead, we choose some smoothing parameter $\alpha > 0$ and, following~\cref{def:smoothed-support} and~\cref{prop:perturbed-opt}, we consider the smoothed minimization problem
\begin{align}
    \label{prob:dual-smoothed}
    \min_{y \in \Eset} \ \big\{ \support_\alpha(-y) + f^*(y)\big\}.
\end{align}
Problem~\eqref{prob:dual-smoothed} fits within the framework of~\eqref{prob:composite} with $G(y) = \support_\alpha(-y)$, $H(y) = f^*(y)$, $\muh = \nicefrac{1}{L}$ and $\Lg = \frac{R_K M}{\alpha}$. Using~\cref{prop:perturbed-opt}, we see that an unbiased stochastic gradient $g$ of $\G$ at some point $y$ can be obtained by sampling $\Delta$ and computing $g = - \argmax_{u \in K} \inner{u}{-y + \alpha \Delta}$. This boils down to a linear optimization oracle over $K$, the same oracle as that of the Frank-Wolfe algorithm.

We now show how to pick the distance generating function $\w$ such that the minimization step of~\cref{alg:acc-stoch-bregman} has a closed form. Recalling that $z_0$ must minimize $\w$, we set $w(y) = f^*(y) - \inner{g_{f^*}(z_0)}{y}$ where $g_{f^*}(z_0) \in \partial f^*(z_0)$. Clearly, $\w$ is $\nicefrac{1}{L}$-strongly convex w.r.t. $\norm{\cdot}$ and is minimized at $z_0$. Moreover, first-order optimality conditions for the minimization step~\eqref{alg1:minimization-step} of~\cref{alg:acc-stoch-bregman} read
\begin{align}
    0 &\in d_{k+1} + A_{k+1} \partial f^*(z_{k+1}) + \frac{R_K M}{\alpha} \left( \partial f^*(z_{k+1}) - g_{f^*}(z_0) \right)\\
    \label{eq:opt-conditions-dual}
     \iff \partial f^*(z_{k+1}) &\ni \frac{\frac{R_K M}{\alpha}}{A_{k+1} + \frac{R_K M}{\alpha}} g_{f^*}(z_0) - \frac{d_{k+1}}{A_{k+1} + \frac{R_K M}{\alpha}}.
\end{align}
Choosing some $x_0 \in K$ and setting $z_k = \nabla f(x_k)$ for all $k$, one can replace~\eqref{eq:opt-conditions-dual} by 
\begin{align}
    x_{k+1} = \frac{\frac{R_K M}{\alpha}}{A_{k+1} + \frac{R_K M}{\alpha}} x_0 - \frac{d_{k+1}}{A_{k+1} + \frac{R_K M}{\alpha}}.
\end{align}
Doing this allows to avoid computing subgradients of $f^*$. Instead, whenever the value of $z_{k}$ is needed, we compute $\nabla f(x_{k})$. We summarize this fully primal algorithm in~\cref{alg:AFW}. Note that similar tricks to obtain ``dual-free'' methods have already been used, see for example~\cite{lan2012optimal, wang2017exploiting, hendrikx2020dual, cohen2021relative, jin2022sharper}. 

\begin{algorithm}[ht]
\caption{Parallel Frank-Wolfe (PFW)  }
\label{alg:AFW}
\textbf{Input}: $(L, x_0, R_K, M, m, \alpha, T)$. $L$-smooth convex function $f$ w.r.t. $\norm{\cdot}_*$, $x_0 \in K$, $R_K = \max_{x\in K} \norm{x}_*$, distribution with density $d\pi(z) \propto \exp(-\eta(z))dz$ such that $M^2 = \E_\Delta \norm{\nabla \eta(\Delta)}_*^2$, $m$ computing units in parallel, smoothing parameter $\alpha$, number of iterations $T$.\\
Set $A_0 = 0$, $d_0 = 0$, $\beta = \tfrac{R_K M}{\alpha}$, $\muh = \muw =  \tfrac{1}{L}$, $y_0 = \nabla f(x_0)$.\\
\For{$k =0, 1, \dots, T-1$}{
\begin{align*}
     A_{k+1} &= \tfrac{A_k(\muh + 2\Lg + \sqrt{\muh \Lg}) + \Lg \muh + \sqrt{\muh^2(\Lg +  A_k)^2 + 4A_k\Lg^2 \muh + 5 A_k^2 \muh \Lg + 2A_k \muh\sqrt{\muh \Lg}( \Lg + A_k) }}{2(\Lg + \sqrt{\muh \Lg})}\\
        \tau_k &=1-\tfrac{A_k}{A_{k+1}}\\
        v_k &= (1 - \tauk)y_k + \tauk \nabla f(x_k)
    \intertext{For all $i \in [m]$ in parallel, sample $\Delta_i \sim d\pi $, compute 
    $g_{k,i} =  -\argmax_{u \in K}\inner{u}{-v_k + \alpha \Delta_i}$}
    g_k &= \tfrac{1}{m} \sum_{i=1}^m g_{k, i}\\
        d_{k+1} &= d_k + (A_{k+1} - A_k) g_k\\
        x_{k+1} &= \tfrac{\beta}{A_{k+1} + \beta} x_0 - \tfrac{1}{A_{k+1} + \beta}d_{k+1}\\
         y_{k+1} &= (1 - \tauk)y_k + \tauk \nabla f(x_{k+1})
    \end{align*}
}
\end{algorithm}

Observe that each iteration of~\cref{alg:AFW} requires the computation of one gradient of $f$, and $m$ calls to a linear maximization oracle. When $m=1$, an iteration of~\cref{alg:AFW} is as costly as one iteration of Frank-Wolfe. We now show how to further exploit parallelization. To do so, we need to ensure that using multiple samples appropriately improves the quality of the approximation of the true gradient. This the point of the following assumption. 
\begin{assumption}
\label{assumption:improved_variance} 
There exists a norm-dependent constant $\rho_{\norm{\cdot}_*}$ such that the variance verifies 
\begin{align}
\sigma^2 = \E_k\normsq{\frac{1}m \sum_{i=1}^m g_{k,i}-\nabla G(v_k)}_* \leq \frac{ 4 R_K^2 \rho_{\|\cdot\|_*}}{m}.
\end{align}
\end{assumption}
\begin{remark}
A few examples are in order. For the standard Euclidean norm $\|\cdot\| = \|\cdot\|_* =\|\cdot\|_2$ it holds that $\rho_{\|\cdot\|_2}=1$. For $\ell_p$-norms on a space of dimension $d$, we have
\begin{align}
    \rho_{\norm{\cdot}_p} = \begin{cases}
    d^{\nicefrac{2}{p} - 1} &\text{ for } 1 \leq p < 2\\
    p-1 & \text{ for } 2 \leq p < \infty\\
    e^2(\log d + 1) &\text{ for } p = \infty.
    \end{cases}
\end{align}
In the general case, as all norms are equivalent in finite dimensional spaces, there exist constants $c, C > 0$ such that $c \norm{\cdot}_2 \leq \norm{\cdot}_* \leq C \norm{\cdot}_2$. Then $\rho_{\norm{\cdot}_*} \leq \frac{C^2}{c^2}$. We refer to Appendix~\ref{app:norm_dependency} for more details and references.
\end{remark}

As mentioned earlier, one can directly obtain convergence of the dual problem~\eqref{prob:dual-smoothed} by applying~\cref{thm:composite-result}. The next theorem states that not only are the iterates $\{x_k\}_{k \in \mathbb{N}}$ always feasible, but also that the previous rate of convergence also holds for the primal-dual gap between the iterates $\{x_k\}_{k\in\mathbb{N}}$ and $\{y_k\}_{k\in\mathbb{N}}$. For simplicity of the exposition we assume that $\frac{R_K M}{\alpha} \geq \frac{1}{L}$. In practice this is not an issue since we will choose small values of $\alpha$. 

\begin{theorem}
\label{thm:AFW-result}
Suppose $f$ is convex and $L$-smooth w.r.t. $\norm{\cdot}_*$ and let $x_0 \in K$. Under~\cref{alg:AFW}, the primal iterates are always feasible, i.e.,  $x_k \in K$ for all $k \in \mathbb{N}$. Moreover, after $\N$ iterations, we have
     \begin{align*}
     %\label{eq:AFW-result-bound}
        \E \left[ f(x_\N) - \D(y_\N) \right] \leq \exp\left( -{\color{red}(\N - 1)} \tfrac{\sqrt{\alpha}}{{\color{red} 8} \sqrt{L R_K M}}\right) {\color{red} 2L }\tfrac{R_K M}{ \alpha} \left( f(x_0) - f(\xstar)\right) + \tfrac{2R_K^2 \rho_{\norm{\cdot}_*}}{m} \sqrt{\tfrac{\alpha L}{R_K M}} + \alpha \support_1(0).
    \end{align*}
\end{theorem}
The above convergence result is obtained by first plugging the smoothness and strong convexity constants of this section within the result of~\cref{prop:Lyap-breg}. Then, one has to relate $m_k(z_k)$ to the primal objective $f(x_k)$. The first term in the upper bound then follows by recursion. The second term comes from the variance, which can be bounded by $\tfrac{4R_K^2 \rho_{\norm{\cdot}}}{m}$ when $m$ stochastic gradients are computed in parallel. The last term $\alpha \support_1(0)$ is due to the error induced by the smoothing of the dual. The full proof can be found in~\cref{app:AFW-proofs}. From~\cref{thm:AFW-result} we are able to show that~\cref{alg:AFW} achieves acceleration under parallelization.

\begin{theorem}
\label{thm:AFW-result-bigO}
Suppose $f$ is convex and $L$-smooth w.r.t. $\norm{\cdot}_*$ and let $x_0 \in K$. Under~\cref{alg:AFW} with $\alpha = \min \left\{ \frac{\epsilon}{3 \support_1(0) }, \frac{ M \epsilon^2 m^2}{36 L R_K^3 \rho_{\norm{\cdot}_*}^2} \right\}$, the number of iterations required to achieve an $\epsilon$ primal-dual gap, i.e., $\E[f(x_\N) - d(y_\N)] \leq \epsilon$, is
\begin{align}
    \label{eq:thm-AFW-big0-1}
    \N \geq {\color{red}1+}\frac{{\color{red}8} \sqrt{L R_K M }}{\sqrt{\alpha}} \log \left( \frac{{\color{red} 6L} R_K M (f(x_0) - f(\xstar))}{\epsilon \alpha}\right).
\end{align}
The complexity is therefore $\tilde{O}\left( \max\left( \frac{1}{\sqrt{\epsilon}}, \frac{1}{\epsilon m}\right)\right).$
%\begin{align}
 %   \label{eq:AFW-complexity}
  %   \N \geq \tilde{O}\left( \max\left( \frac{1}{\sqrt{\epsilon}}, \frac{1}{\epsilon m}\right)\right).
%\end{align}
\end{theorem}
The above complexity result shows that (i)~when $m = \nicefrac{1}{\sqrt{\epsilon}}$, we get an accelerated rate $\tilde{O}(\nicefrac{1}{ \sqrt{\epsilon}})$, (ii)~when $m=1$, we recover the classical $\tilde{O}(\nicefrac{1}{\epsilon})$ complexity of Frank-Wolfe (up to logarithmic terms), (iii)~there is no theoretical gain in going beyond $m = \nicefrac{1}{\sqrt{\epsilon}}$ computing units in parallel.

Moreover, while the total number of calls to the linear optimization oracle to reach an $\epsilon$ primal-dual gap is the same as in Frank-Wolfe for any value of $m$, in the case where $m > 1$, the number of required gradients of $f$  is strictly less than in Frank-Wolfe. This is because in each iteration of~\cref{alg:AFW} we compute one gradient of $f$ no matter what the value of $m$ is. In particular, when $m = \nicefrac{1}{\sqrt{\epsilon}}$, we need only $\tilde{O}(\nicefrac{1}{\sqrt{\epsilon}})$ gradients of $f$ in total, compared to $O(\nicefrac{1}{\epsilon})$ in Frank-Wolfe.

One can note that the proposed algorithm in not as universal as the classical Frank-Wolfe method, as it requires upper bounds on several problem-specific constants. However, acceleration of Frank-Wolfe has been extensively studied in the literature (see~\cref{sec:introduction}) and to the best of our knowledge our work is the first to provide accelerated rates (under parallelization) without further assumptions on the constraint set and/or the objective function. We argue next that estimating upper bounds on most parameters does not pose significant challenges. Indeed, as the user is free to pick any distribution to sample $\Delta$ from (up to the assumptions of~\cref{prop:perturbed-opt}), one can choose a distribution for which the constant $M$ is easy to compute (we give two examples in~\cref{sec:experiments}). Moreover, the user typically knows the set over which the optimization is carried over, and from such knowledge an upper bound on the diameter can often be easily computed. Finally, we also need to upper bound the Lipschitz constant $L$ of the gradient. This seems to be more of a limiting factor compared to classical Frank-Wolfe, although an upper bound on $L$ is often (but not always) necessary to implement other accelerated versions of Frank-Wolfe \cite{garber2015faster, lan2016conditional} as well. While it might be possible to circumvent this problem using some sort of line search within our method, it is not straightforward as the smoothing of the dual induces stochasticity, and obtaining theoretical rates for line search techniques on stochastic objectives is notoriously tedious \cite{bottou2018optimization, vaswani2019painless, vaswani2022towards}.

Finally, note that it might seem like a small modification to~\cref{alg:AFW} could give a variant of Frank-Wolfe capable of handling stochastic gradients of~$f$ (for previous works on stochastic Frank-Wolfe, see \cite{hazan2016variance, ghadimi2019conditional, negiar2020stochastic} and references therein). Indeed, one could see the noise in $\nabla f(x_k)$ in the update of $v_k$ as the random variable $\Delta$ itself and study the resulting algorithm. However, the current analysis would only work if the noise on $\nabla f(x_k)$ does not depend on the current iterate. Moreover, even under this simplifying assumption, the result does not follow immediately as the stochasticity would appear in both the updates of $v_k$ and $y_{k+1}$, in contrast with the current analysis.

\section{Experiments}
\label{sec:experiments}
We first consider minimizing a quadratic function in $\R^d$ over the simplex
\begin{align}
    \label{eq:simplex}
    \min_{x \in K_1} \ \{f_1(x) := \frac{1}{2} \normsq{Ax - b}_2\} \quad \text{where} \quad K_1 = \left\{ x\in \R^d \mid x \geq 0, \sum_{i=1}^d x_i = 1\right\},
\end{align}
where $A \in \R^{n \times d}$ and $b\in \R^n$. We choose $\Delta$ to have the Gumbel distribution~\cite{gumbel1954statistical} with location and scale parameters equal to 0 and 1 respectively. In this case $R_{K_1} = 1$ and $M = \sqrt{d}$ (see~\cref{app:exp-details} for details). We set $n=200$, $d=50$ and compare the bound from~\cref{thm:AFW-result} with the practical performance of~\cref{alg:AFW} for both $m=1$ and $m= \nicefrac{1}{\sqrt{\alpha}}$ parallel computer(s).~\cref{figure1:upper-bound-comparison_New} shows that our upper bound captures well the speed of convergence to a neighborhood.

\begin{figure}[!ht]
 \includegraphics[scale=0.36]{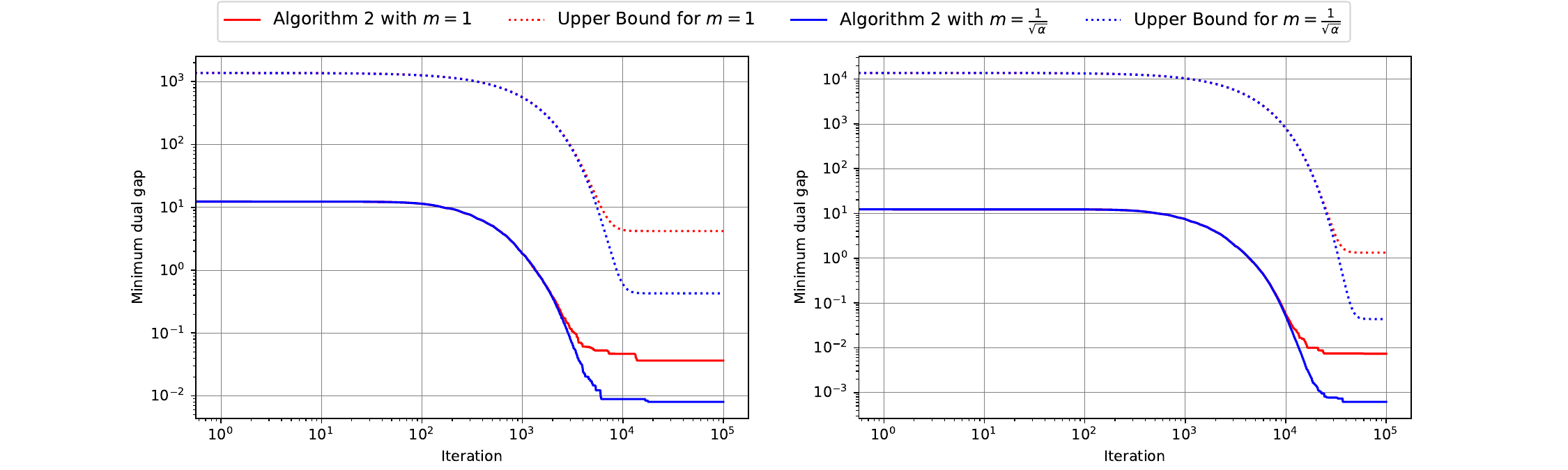}
\caption{Comparisons between the behavior of \Cref{alg:AFW} and that of its theoretical upper bound (see \Cref{thm:AFW-result}) on a least-squares problem on the simplex for $\alpha = 10^{-2}$ (left) and $\alpha = 10^{-3}$ (right). The plots report the value of the best primal-dual gap incurred at the current iteration.}
  \label{figure1:upper-bound-comparison_New}
\end{figure}

Next, in order to circumvent the inevitable stalling due to the fixed value of $\alpha$ observed above, we suggest a restarted algorithm which decreases the value of $\alpha$ during training. Starting from $\alpha = 1$ and some $x_0 \in K$, we run~\cref{alg:AFW} for $T_\alpha$ iterations with $m_\alpha$ computers in parallel to obtain some approximate solution $x_{\alpha} \in K$. We then decrease $\alpha$ by a constant factor $c < 1$ and run~\cref{alg:AFW} again with the new value of $\alpha$ starting from $x_{\alpha}$. We repeat this process until a satisfying solution is obtained. We formalize this procedure in~\cref{alg:R-PFW}. In practice we choose $c = 0.5$ and set $T_\alpha = \sqrt{\frac{L}{\alpha}} \log \frac{1}{\alpha}$ as a simplified version of the bound~\eqref{eq:thm-AFW-big0-1}, and run the above procedure for both $m_\alpha = 1$ and $m_\alpha = \nicefrac{1}{\sqrt{\alpha}}$ computers in parallel.

\begin{algorithm}[H]
\caption{Restarted Parallel Frank-Wolfe (R-PFW)}
\label{alg:R-PFW}
\textbf{Input}: $(L, x_0, R_K, M, c)$. $L$-smooth convex function $f$, $x_0 \in K$, $R_K = \max_{x\in K} \norm{x}_*$, distribution with density $d\pi(z) \propto \exp(-\eta(z))dz$ such that $M^2 = \E_\Delta \norm{\nabla \eta(\Delta)}_*^2$, decreasing factor $c$.\\
Set $\alpha = 1$.\\
\For{$i = 0, 1, \dots$}{
     Set $T_\alpha = \sqrt{\frac{L}{\alpha}} \log \frac{1}{\alpha}$ and either set $m_\alpha = \frac{1}{\sqrt{\alpha}}$ or $m_\alpha = 1$.\\
     $x_{\alpha} = \text{\cref{alg:AFW}}(L, x_0, R_K, M, m_\alpha, \alpha, T_\alpha )$\vspace{0.5ex}\\
     Set $x_0 = x_\alpha$ and $\alpha = c\alpha$.
}
Return $x_\alpha$.
\end{algorithm}

We test the  restarted scheme on problem~\eqref{eq:simplex} as well as on a generalized matrix completion problem over the trace norm ball,
\begin{align}
    \label{eq:matrix-completion}
    \min_{X \in K_2} \ \{f_2(x) := \frac{1}{2} \normsq{CX - D}_F\} \quad \text{where} \quad K_2 = \left\{ X\in \R^{p \times q} \mid \norm{X}_{tr} \leq 1\right\} ,
\end{align}
where $C\in \R^{p \times p}$ and $D \in \R^{p \times q}$. Here $\norm{\cdot}_F$ stands for the Frobenius norm and $\norm{\cdot}_{tr}$ is the trace, or nuclear, norm. We choose $\Delta$ to have entries distributed according to a standard normal distribution. In this case $R_{K_2} =1$ and $M = \sqrt{pq}$.  We compare the restarted scheme with the Frank-Wolfe algorithm with step-size $\eta_k = \frac{2}{k+1}$ (denoted FW in the plots) and with exact line search (denoted FW-LS in the plots). We set $p = 10$, $q=8$ and plot the results in~\cref{figure:R-PFW_New}. For both problems, we observed that using $M = 1$ instead of the theoretical value ($M = \sqrt{d}$ for~\eqref{eq:simplex} and $M = \sqrt{pq}$ for~\eqref{eq:matrix-completion}) led to significant speedups. We consequently plot both strategies. We observe that for $m_\alpha = \nicefrac{1}{\sqrt{\alpha}}$, \cref{alg:R-PFW} significantly outperforms Frank-Wolfe, and for $m_\alpha = 1$ the performance is similar to Frank-Wolfe, which was expected from the discussion at the end of~\cref{sec:AFW}. Full code to reproduce the experiments can be found at the following link: \url{https://github.com/bpauld/PFW}.

\begin{figure}[!ht]
\centering
 \includegraphics[scale=0.54]{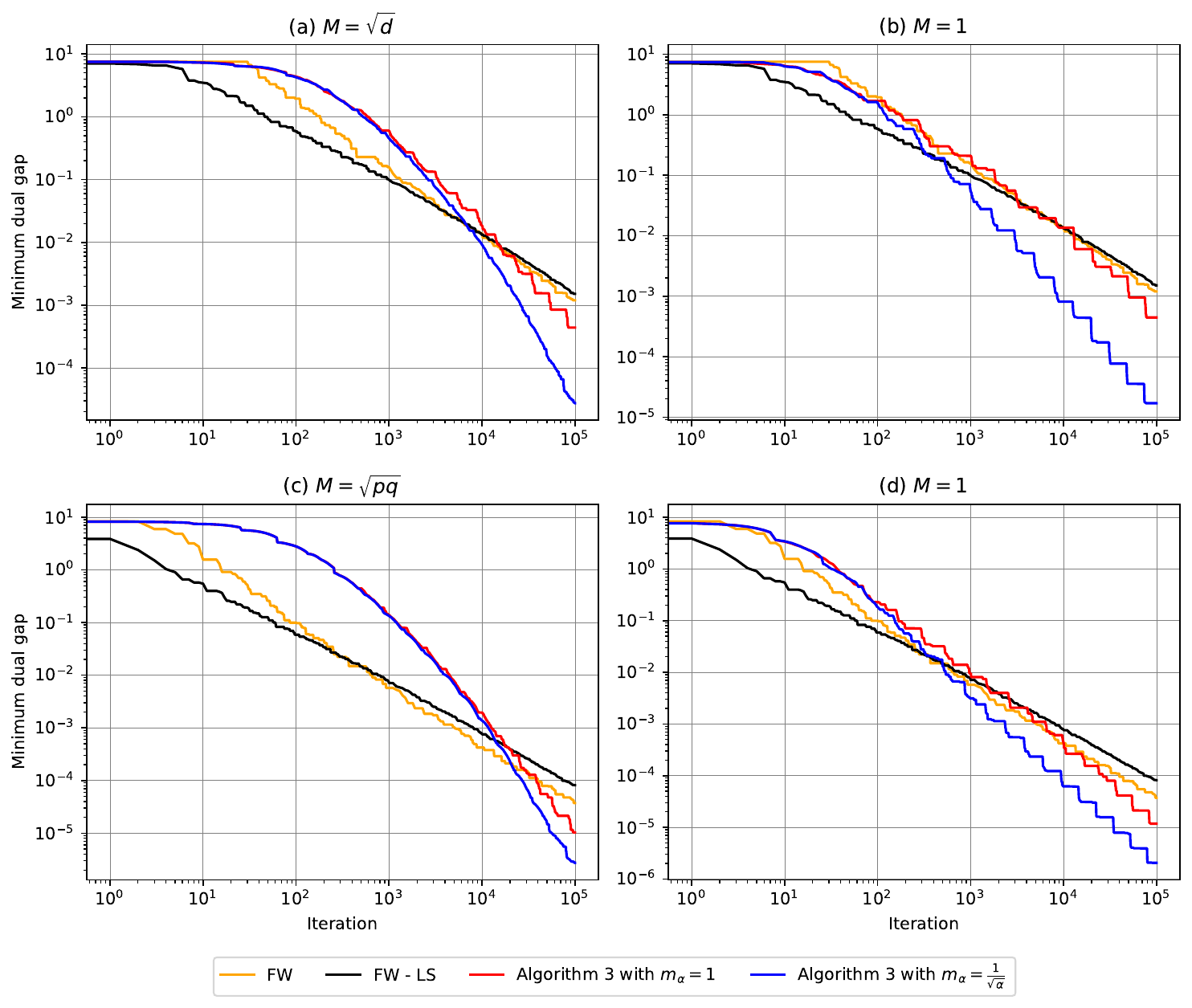}
\caption{Comparisons between Frank-Wolfe and the restarting scheme (\Cref{alg:R-PFW}): a least-squares problem on the simplex ((a) and (b)), and a generalized matrix completion problem on the trace ball ((c) and (d)). The plots report the value of the best primal-dual gap incurred at the current iteration.}
  \label{figure:R-PFW_New}
\end{figure}

\section{Conclusion}
In this work we introduced a stochastic proximal Bregman algorithm able to converge to a neighborhood of the solution at the same accelerated rate as its deterministic counterpart. We then used it to design a variant of the Frank-Wolfe algorithm able to achieve accelerated rates under parallelization. One drawback is that the resulting algorithm is not "any time", in the sense that the desired precision~$\epsilon$ must be given as an input to the algorithm. We circumvent this by designing a simple heuristic which slowly decreases the value of~$\epsilon$. We plan to investigate the theory for this restarted algorithm in the future. We also hope to use randomized smoothing techniques similar to the ones used in the work to obtain stochastic Frank-Wolfe methods.
\section{Acknowledgements}

The authors acknowledge support from the European Research Council (grant SEQUOIA 724063). This work was funded in
part by the french government under management of Agence Nationale de la recherche as part of the “Investissements
d’avenir” program, reference ANR-19-P3IA-0001 (PRAIRIE 3IA Institute).

The authors would like to warmly thank Alexander Gasnikov for very constructive remarks and for pointing out the existence and relationship of our work with~\cite{lan2016conditional, ghadimi2019conditional, jin2022sharper, cohen2021relative}.
%\clearpage
\bibliography{ref} 

\begin{thebibliography}{46}
\providecommand{\natexlab}[1]{#1}
\providecommand{\url}[1]{\texttt{#1}}
\expandafter\ifx\csname urlstyle\endcsname\relax
  \providecommand{\doi}[1]{doi: #1}\else
  \providecommand{\doi}{doi: \begingroup \urlstyle{rm}\Url}\fi

\bibitem[Abernethy et~al.(2014)Abernethy, Lee, Sinha, and
  Tewari]{abernethy2014online}
Jacob Abernethy, Chansoo Lee, Abhinav Sinha, and Ambuj Tewari.
\newblock Online linear optimization via smoothing.
\newblock In \emph{Conference on Learning Theory (COLT)}, 2014.

\bibitem[Abernethy et~al.(2016)Abernethy, Lee, and
  Tewari]{abernethy2016perturbation}
Jacob Abernethy, Chansoo Lee, and Ambuj Tewari.
\newblock Perturbation techniques in online learning and optimization.
\newblock \emph{Perturbations, Optimization, and Statistics}, page 223, 2016.

\bibitem[Aybat et~al.(2019)Aybat, Fallah, Gurbuzbalaban, and
  Ozdaglar]{aybat2019universally}
Necdet~S. Aybat, Alireza Fallah, Mert Gurbuzbalaban, and Asuman Ozdaglar.
\newblock A universally optimal multistage accelerated stochastic gradient
  method.
\newblock In \emph{Advances in Neural Information Processing Systems
  (NeurIPS)}, 2019.

\bibitem[Bach(2015)]{bach2015duality}
Francis Bach.
\newblock Duality between subgradient and conditional gradient methods.
\newblock \emph{SIAM Journal on Optimization}, 25\penalty0 (1):\penalty0
  115--129, 2015.

\bibitem[Bach(2021)]{bach2020effectiveness}
Francis Bach.
\newblock On the effectiveness of {R}ichardson extrapolation in data science.
\newblock \emph{SIAM Journal on Mathematics of Data Science}, 3\penalty0
  (4):\penalty0 1251--1277, 2021.

\bibitem[Ball et~al.(1994)Ball, Carlen, and Lieb]{ball1994sharp}
Keith Ball, Eric~A Carlen, and Elliott~H Lieb.
\newblock Sharp uniform convexity and smoothness inequalities for trace norms.
\newblock \emph{Inventiones mathematicae}, 115\penalty0 (1):\penalty0 463--482,
  1994.

\bibitem[Berthet et~al.(2020)Berthet, Blondel, Teboul, Cuturi, Vert, and
  Bach]{berthet2020learning}
Quentin Berthet, Mathieu Blondel, Olivier Teboul, Marco Cuturi, Jean-Philippe
  Vert, and Francis Bach.
\newblock Learning with differentiable pertubed optimizers.
\newblock In \emph{Advances in Neural Information Processing Systems
  (NeurIPS)}, 2020.

\bibitem[Blondel et~al.(2021)Blondel, Berthet, Cuturi, Frostig, Hoyer,
  Llinares-L{\'o}pez, Pedregosa, and Vert]{jaxopt_implicit_diff}
Mathieu Blondel, Quentin Berthet, Marco Cuturi, Roy Frostig, Stephan Hoyer,
  Felipe Llinares-L{\'o}pez, Fabian Pedregosa, and Jean-Philippe Vert.
\newblock Efficient and modular implicit differentiation.
\newblock \emph{arXiv preprint arXiv:2105.15183}, 2021.

\bibitem[Bottou et~al.(2018)Bottou, Curtis, and
  Nocedal]{bottou2018optimization}
L{\'e}on Bottou, Frank~E Curtis, and Jorge Nocedal.
\newblock Optimization methods for large-scale machine learning.
\newblock \emph{Siam Review}, 60\penalty0 (2):\penalty0 223--311, 2018.

\bibitem[Cohen et~al.(2021)Cohen, Sidford, and Tian]{cohen2021relative}
Michael~B Cohen, Aaron Sidford, and Kevin Tian.
\newblock Relative lipschitzness in extragradient methods and a direct recipe
  for acceleration.
\newblock In \emph{12th Innovations in Theoretical Computer Science Conference
  (ITCS 2021)}, volume 185, 2021.

\bibitem[Demyanov and Rubinov(1970)]{demianov1970approximate}
Vladimir~Fedorovich Demyanov and Aleksandr~Moiseevich Rubinov.
\newblock \emph{Approximate methods in optimization problems}.
\newblock Elsevier Publishing Company, 1970.

\bibitem[Devolder(2013)]{devolder2013exactness}
Olivier Devolder.
\newblock \emph{Exactness, inexactness and stochasticity in first-order methods
  for large-scale convex optimization}.
\newblock PhD thesis, PhD thesis, ICTEAM and CORE, Universit{\'e} Catholique de
  Louvain, 2013.

\bibitem[Diakonikolas and Guzm{\'a}n(2021)]{diakonikolas2021complementary}
Jelena Diakonikolas and Crist{\'o}bal Guzm{\'a}n.
\newblock Complementary composite minimization, small gradients in general
  norms, and applications to regression problems.
\newblock \emph{arXiv preprint 2101.11041}, 2021.

\bibitem[Duchi et~al.(2012)Duchi, Bartlett, and
  Wainwright]{duchi2012randomized}
John~C. Duchi, Peter~L. Bartlett, and Martin~J. Wainwright.
\newblock Randomized smoothing for stochastic optimization.
\newblock \emph{SIAM Journal on Optimization}, 22\penalty0 (2):\penalty0
  674--701, 2012.

\bibitem[d’Aspremont et~al.(2021)d’Aspremont, Scieur, and
  Taylor]{d2021acceleration}
Alexandre d’Aspremont, Damien Scieur, and Adrien Taylor.
\newblock Acceleration methods.
\newblock \emph{Foundations and Trends{\textregistered} in Optimization},
  5\penalty0 (1-2):\penalty0 1--245, 2021.

\bibitem[Frank and Wolfe(1956)]{frank1956algorithm}
Marguerite Frank and Philip Wolfe.
\newblock An algorithm for quadratic programming.
\newblock \emph{Naval research logistics quarterly}, 3\penalty0 (1-2):\penalty0
  95--110, 1956.

\bibitem[Freund et~al.(2017)Freund, Grigas, and Mazumder]{freund2017extended}
Robert~M. Freund, Paul Grigas, and Rahul Mazumder.
\newblock An extended {F}rank--{W}olfe method with “in-face” directions,
  and its application to low-rank matrix completion.
\newblock \emph{SIAM Journal on optimization}, 27\penalty0 (1):\penalty0
  319--346, 2017.

\bibitem[Garber and Hazan(2015)]{garber2015faster}
Dan Garber and Elad Hazan.
\newblock Faster rates for the frank-wolfe method over strongly-convex sets.
\newblock In \emph{International Conference on Machine Learning}, pages
  541--549. PMLR, 2015.

\bibitem[Gasnikov and Nesterov(2018)]{gasnikov2018universal}
Alexander~V. Gasnikov and Yurii Nesterov.
\newblock Universal method for stochastic composite optimization problems.
\newblock \emph{Computational Mathematics and Mathematical Physics},
  58\penalty0 (1):\penalty0 48--64, 2018.

\bibitem[Ghadimi(2019)]{ghadimi2019conditional}
Saeed Ghadimi.
\newblock Conditional gradient type methods for composite nonlinear and
  stochastic optimization.
\newblock \emph{Mathematical Programming}, 173\penalty0 (1):\penalty0 431--464,
  2019.

\bibitem[Gu{\'e}lat and Marcotte(1986)]{guelat1986some}
Jacques Gu{\'e}lat and Patrice Marcotte.
\newblock Some comments on {W}olfe's ‘away step’.
\newblock \emph{Mathematical Programming}, 35\penalty0 (1):\penalty0 110--119,
  1986.

\bibitem[Gumbel(1954)]{gumbel1954statistical}
Emil~J. Gumbel.
\newblock \emph{Statistical theory of extreme values and some practical
  applications: a series of lectures}, volume~33.
\newblock US Government Printing Office, 1954.

\bibitem[Harchaoui et~al.(2015)Harchaoui, Juditsky, and
  Nemirovski]{harchaoui2015conditional}
Zaid Harchaoui, Anatoli Juditsky, and Arkadi Nemirovski.
\newblock Conditional gradient algorithms for norm-regularized smooth convex
  optimization.
\newblock \emph{Mathematical Programming}, 152\penalty0 (1):\penalty0 75--112,
  2015.

\bibitem[Hazan and Luo(2016)]{hazan2016variance}
Elad Hazan and Haipeng Luo.
\newblock Variance-reduced and projection-free stochastic optimization.
\newblock In \emph{International Conference on Machine Learning}, pages
  1263--1271. PMLR, 2016.

\bibitem[Hendrikx et~al.(2020)Hendrikx, Bach, and
  Massouli{\'e}]{hendrikx2020dual}
Hadrien Hendrikx, Francis Bach, and Laurent Massouli{\'e}.
\newblock Dual-free stochastic decentralized optimization with variance
  reduction.
\newblock \emph{Advances in Neural Information Processing Systems (NeurIPS)},
  2020.

\bibitem[Jaggi(2011)]{jaggi2011sparse}
Martin Jaggi.
\newblock \emph{Sparse Convex Optimization Methods for Machine Learning}.
\newblock PhD thesis, ETH Zurich, 2011.

\bibitem[Jaggi(2013)]{jaggi2013revisiting}
Martin Jaggi.
\newblock Revisiting frank-wolfe: Projection-free sparse convex optimization.
\newblock In \emph{International Conference on Machine Learning (ICML)}, 2013.

\bibitem[Jin et~al.(2022)Jin, Sidford, and Tian]{jin2022sharper}
Yujia Jin, Aaron Sidford, and Kevin Tian.
\newblock Sharper rates for separable minimax and finite sum optimization via
  primal-dual extragradient methods.
\newblock \emph{arXiv preprint arXiv:2202.04640}, 2022.

\bibitem[Kakade et~al.(2009)Kakade, Shalev-Shwartz, Tewari,
  et~al.]{kakade2009duality}
Sham Kakade, Shai Shalev-Shwartz, Ambuj Tewari, et~al.
\newblock On the duality of strong convexity and strong smoothness: Learning
  applications and matrix regularization.
\newblock \emph{Unpublished Manuscript, http://ttic. uchicago.
  edu/shai/papers/KakadeShalevTewari09. pdf}, 2\penalty0 (1), 2009.

\bibitem[Kakade et~al.(2008)Kakade, Sridharan, and
  Tewari]{kakade2008complexity}
Sham~M Kakade, Karthik Sridharan, and Ambuj Tewari.
\newblock On the complexity of linear prediction: Risk bounds, margin bounds,
  and regularization.
\newblock \emph{Advances in neural information processing systems}, 21, 2008.

\bibitem[Lacoste-Julien and Jaggi(2015)]{lacoste2015global}
Simon Lacoste-Julien and Martin Jaggi.
\newblock On the global linear convergence of frank-wolfe optimization
  variants.
\newblock \emph{Advances in Neural Information Processing Systems (NIPS)},
  2015.

\bibitem[Lan(2012)]{lan2012optimal}
Guanghui Lan.
\newblock An optimal method for stochastic composite optimization.
\newblock \emph{Mathematical Programming}, 133\penalty0 (1):\penalty0 365--397,
  2012.

\bibitem[Lan(2013)]{lan2013complexity}
Guanghui Lan.
\newblock The complexity of large-scale convex programming under a linear
  optimization oracle.
\newblock \emph{arXiv preprint 1309.5550}, 2013.

\bibitem[Lan and Zhou(2016)]{lan2016conditional}
Guanghui Lan and Yi~Zhou.
\newblock Conditional gradient sliding for convex optimization.
\newblock \emph{SIAM Journal on Optimization}, 26\penalty0 (2):\penalty0
  1379--1409, 2016.

\bibitem[Levitin and Polyak(1966)]{levitin1966constrained}
Evgeny~S. Levitin and Boris~T. Polyak.
\newblock Constrained minimization methods.
\newblock \emph{USSR Computational mathematics and mathematical physics},
  6\penalty0 (5):\penalty0 1--50, 1966.

\bibitem[Li et~al.(2020)Li, Coutino, Giannakis, and Leus]{li2020does}
Bingcong Li, Mario Coutino, Georgios~B. Giannakis, and Geert Leus.
\newblock How does momentum help {F}rank {W}olfe?
\newblock \emph{arXiv preprint 2006.11116}, 2020.

\bibitem[N{\'e}giar et~al.(2020)N{\'e}giar, Dresdner, Tsai, El~Ghaoui,
  Locatello, Freund, and Pedregosa]{negiar2020stochastic}
Geoffrey N{\'e}giar, Gideon Dresdner, Alicia Tsai, Laurent El~Ghaoui, Francesco
  Locatello, Robert Freund, and Fabian Pedregosa.
\newblock Stochastic frank-wolfe for constrained finite-sum minimization.
\newblock In \emph{International Conference on Machine Learning}, pages
  7253--7262. PMLR, 2020.

\bibitem[Nesterov(1983)]{nesterov1983method}
Yurii Nesterov.
\newblock A method for unconstrained convex minimization problem with the rate
  of convergence ${O}(1/k^2)$.
\newblock In \emph{Doklady an USSR}, volume 269, pages 543--547, 1983.

\bibitem[Nesterov(2003)]{nesterov2003introductory}
Yurii Nesterov.
\newblock \emph{Introductory Lectures on Convex Optimization: A Basic Course}.
\newblock Springer Science \& Business Media, 2003.

\bibitem[Nesterov(2013)]{nesterov2013gradient}
Yurii Nesterov.
\newblock Gradient methods for minimizing composite functions.
\newblock \emph{Mathematical Programming}, 140\penalty0 (1):\penalty0 125--161,
  2013.

\bibitem[Nesterov and Spokoiny(2017)]{nesterov2017random}
Yurii Nesterov and Vladimir Spokoiny.
\newblock Random gradient-free minimization of convex functions.
\newblock \emph{Foundations of Computational Mathematics}, 17\penalty0
  (2):\penalty0 527--566, 2017.

\bibitem[Rockafellar(2015)]{rockafellar2015convex}
Ralph~T. Rockafellar.
\newblock \emph{Convex Analysis}.
\newblock Princeton University Press, 2015.

\bibitem[Rockafellar and Wets(2009)]{rockafellar2009variational}
Ralph~T. Rockafellar and Roger J.-B. Wets.
\newblock \emph{Variational Analysis}.
\newblock Springer Science \& Business Media, 2009.

\bibitem[Vaswani et~al.(2019)Vaswani, Mishkin, Laradji, Schmidt, Gidel, and
  Lacoste-Julien]{vaswani2019painless}
Sharan Vaswani, Aaron Mishkin, Issam Laradji, Mark Schmidt, Gauthier Gidel, and
  Simon Lacoste-Julien.
\newblock Painless stochastic gradient: Interpolation, line-search, and
  convergence rates.
\newblock \emph{Advances in neural information processing systems}, 32, 2019.

\bibitem[Vaswani et~al.(2022)Vaswani, Dubois-Taine, and
  Babanezhad]{vaswani2022towards}
Sharan Vaswani, Benjamin Dubois-Taine, and Reza Babanezhad.
\newblock Towards noise-adaptive, problem-adaptive (accelerated) stochastic
  gradient descent.
\newblock In \emph{International Conference on Machine Learning}, pages
  22015--22059. PMLR, 2022.

\bibitem[Wang and Xiao(2017)]{wang2017exploiting}
Jialei Wang and Lin Xiao.
\newblock Exploiting strong convexity from data with primal-dual first-order
  algorithms.
\newblock In \emph{International Conference on Machine Learning (ICML)}, 2017.

\end{thebibliography}
\bibliographystyle{plainnat}
%\clearpage
\section*{Checklist}

\begin{enumerate}

\item For all authors...
\begin{enumerate}
  \item Do the main claims made in the abstract and introduction accurately reflect the paper's contributions and scope?
    \answerYes{}
  \item Did you describe the limitations of your work?
    \answerYes{See conclusion.}
  \item Did you discuss any potential negative societal impacts of your work?
    \answerNA{}
  \item Have you read the ethics review guidelines and ensured that your paper conforms to them?
    \answerYes{}
\end{enumerate}

\item If you are including theoretical results...
\begin{enumerate}
  \item Did you state the full set of assumptions of all theoretical results?
    \answerYes{See proposition and theorem statements.}
        \item Did you include complete proofs of all theoretical results?
    \answerYes{See Appendix.}
\end{enumerate}

\item If you ran experiments...
\begin{enumerate}
  \item Did you include the code, data, and instructions needed to reproduce the main experimental results (either in the supplemental material or as a URL)?
    \answerYes{The code is available at \url{https://github.com/bpauld/PFW}.}
  \item Did you specify all the training details (e.g., data splits, hyperparameters, how they were chosen)?
    \answerYes{}
        \item Did you report error bars (e.g., with respect to the random seed after running experiments multiple times)?
    \answerNA{}
        \item Did you include the total amount of compute and the type of resources used (e.g., type of GPUs, internal cluster, or cloud provider)?
    \answerNA{}
\end{enumerate}

\item If you are using existing assets (e.g., code, data, models) or curating/releasing new assets...
\begin{enumerate}
  \item If your work uses existing assets, did you cite the creators?
    \answerNA{}
  \item Did you mention the license of the assets?    \answerNA{}
  \item Did you include any new assets either in the supplemental material or as a URL? \answerNA{}
  \item Did you discuss whether and how consent was obtained from people whose data you're using/curating?    \answerNA{}
  \item Did you discuss whether the data you are using/curating contains personally identifiable information or offensive content?
    \answerNA{}
\end{enumerate}

\item If you used crowdsourcing or conducted research with human subjects...
\begin{enumerate}
  \item Did you include the full text of instructions given to participants and screenshots, if applicable?
    \answerNA{}
  \item Did you describe any potential participant risks, with links to Institutional Review Board (IRB) approvals, if applicable?
    \answerNA{}
  \item Did you include the estimated hourly wage paid to participants and the total amount spent on participant compensation?
    \answerNA{}
\end{enumerate}

\end{enumerate}

%%%%%%%%%%%%%%%%%%%%%%%%%%%%%%%%%%%%%%%%%%%%%%%%%%%%%%%%%%%%

\clearpage
\appendix

\section*{Organization of the Appendix}

\begin{itemize}

   \item[\ref{app:composite-proofs}] 
   \hyperref[app:composite-proofs]{Proofs for~\cref{sec:stochastic-composite-minimization}}
   
   \item[\ref{app:AFW-proofs}] 
   \hyperref[app:AFW-proofs]{Proofs for~\cref{sec:AFW}}
   
   \item[\ref{app:exp-details}] 
   \hyperref[app:exp-details]{Experimental details}
\end{itemize}
\section{Proofs for Stochastic Composite Minimization}
\label{app:composite-proofs}

\subsection{Proof of\texorpdfstring{~\cref{prop:Lyap-breg}}{\space}}
We break down the proof of~\cref{prop:Lyap-breg} into different lemmas. The first one bounds the improvement on $\G$ over an iteration, the second one bounds the improvement $\h$ over an iteration, the third one combines the first two results and further exploits the structure of the functions at play. The proof of the proposition then follows by carefully choosing $A_{k+1}$ as a function of $A_k$. In the remainder of the section, we work with the assumptions stated in~\cref{prop:Lyap-breg} and do not restate them in the statements of the lemmas.

\begin{lemma}
\begin{align*}
    A_{k+1}\G(y_{k+1}) - c_{k+1}  \leq & \,  A_k \G(y_k) - c_k + \left( A_{k+1} - A_k\right)\inner{\nabla \G(v_k)}{z_{k+1}} \\
    & + \frac{\Lg}{2}\left( A_{k+1} - 2A_k + \frac{A_k^2}{A_{k+1}}\right) \normsq{z_{k+1} - z_k}.
\end{align*}
\end{lemma}
\begin{proof}
By smoothness
\begin{align*}
    A_{k+1}\G(y_{k+1}) \leq &\, A_{k+1} \left( \G(v_k) + \inner{\nabla \G(v_k)}{y_{k+1} - v_k} + \frac{\Lg}{2} \normsq{y_{k+1} - v_k} \right)\\
    = &\, \underbrace{\left( A_{k+1} - A_k\right) \left( \G(v_k) - \inner{\nabla \G(v_k)}{v_k} \right)}_{= c_{k+1} - c_k} + A_k\left( \G(v_k) - \inner{\nabla \G(v_k)}{v_k} \right) \\
    & + A_{k+1} \inner{\nabla \G(v_k)}{y_{k+1}} + A_{k+1} \frac{\Lg}{2} \normsq{y_{k+1} - v_k}\\
    = &\, c_{k+1} - c_k + A_k\left( \G(v_k) + \inner{\nabla \G(v_k)}{y_k - v_k} \right) - A_k \inner{\nabla \G(v_k)}{y_k}\\
    & + A_{k+1} \inner{\nabla \G(v_k)}{y_{k+1}} + A_{k+1} \frac{\Lg}{2} \normsq{y_{k+1} - v_k}\\
    \leq &\, c_{k+1} - c_k + A_k \G(y_k) - A_k \inner{\nabla \G(v_k)}{y_k}\\
    & + A_{k+1} \inner{\nabla \G(v_k)}{y_{k+1}} + A_{k+1} \frac{\Lg}{2} \normsq{y_{k+1} - v_k}
\end{align*}
where the last inequality follows by strong convexity of $\G$, Now, since $y_{k+1} = \frac{A_k}{A_{k+1}}y_k + \frac{A_{k+1} - A_k}{A_{k+1}}z_{k+1}$, we can simplify to
\begin{align*}
    A_{k+1}\G(y_{k+1}) &\leq c_{k+1} - c_k + A_k \G(y_k) + \left( A_{k+1} - A_k\right)\inner{\nabla \G(v_k)}{z_{k+1}} + A_{k+1} \frac{\Lg}{2} \normsq{y_{k+1} - v_k}.
\end{align*}
Finally,
\begin{align*}
    A_{k+1} \normsq{y_{k+1} - v_k} = A_{k+1} \tau_k^2 \normsq{z_k - z_{k+1}} = \frac{(A_{k+1} - A_k)^2}{A_{k+1}} \normsq{z_{k+1} - z_k}
\end{align*}
which concludes the proof.
\end{proof}

\begin{lemma}
For any $g_\h(z_k) \in \partial \h(z_k)$,
\begin{align*}
A_{k+1} \h(y_{k+1}) - A_{k+1}\h(z_{k+1}) \leq& \, A_k \h(y_k) - A_k \h(z_k) - A_k \inner{g_\h(z_k)}{z_{k+1} - z_k}\\
& - \frac{\muh}{2} A_k \left( 1- \frac{A_k}{A_{k+1}}\right) \normsq{z_{k+1} - y_k} - A_k \frac{\muh}{2} \normsq{z_{k+1} - z_k}.
\end{align*}
\end{lemma}

\begin{proof}
By strong convexity of $\h$, for any $g_\h(y_{k+1}) \in \partial \h(y_{k+1})$ we have
\begin{align*}
    A_{k+1} \h(y_{k+1}) =&\, \left( A_{k+1} - A_k\right) \h(y_{k+1}) + A_k \h(y_{k+1})\\
    \leq&\, (A_{k+1} - A_k) \left( \h(z_{k+1}) - \inner{g_\h(y_{k+1})}{z_{k+1} - y_{k+1}} - \frac{\muh}{2}\normsq{z_{k+1} - y_{k+1}}\right)\\
    &+ A_k\left( \h(y_k) - \inner{g_\h(y_{k+1})}{y_k - y_{k+1}} - \frac{\muh}{2}\normsq{y_k - y_{k+1}}\right).
\end{align*}
Now observe that
\begin{align*}
    z_{k+1} - y_{k+1} = (1 - \tau_k) (z_{k+1} - y_k) =  \frac{A_k}{A_{k+1}}(z_{k+1} - y_k)\\
    y_k - y_{k+1} = \tau_k (y_k - z_{k+1}) = -\frac{A_{k+1} - A_k}{A_{k+1}}(z_{k+1} - y_k)
\end{align*}
and thus we see that the two inner product terms cancel out. Moreover we can also simplify the norms and get
\begin{align*}
    A_{k+1}\h(y_{k+1}) \leq & A_k \h(y_k) + (A_{k+1} - A_k) \h(z_{k+1}) - \frac{\muh}{2} \normsq{z_{k+1} - y_k} \left( \underbrace{(A_{k+1} - A_k) \tfrac{A_k^2}{A_{k+1}^2}+ A_k \tfrac{(A_{k+1} - A_k)^2}{A_{k+1}^2}}_{= A_k - \tfrac{A_k^2}{A_{k+1}}} \right).
\end{align*}
Finally, observe that by strong convexity of $\h$ again, for any $g_\h(z_k) \in \partial H(z_k)$,
\begin{align*}
    -A_k \h(z_{k+1}) &\leq -A_k \h(z_k) - A_k \inner{g_\h(z_k)}{z_{k+1} - z_k} - A_k \frac{\muh}{2} \normsq{z_{k+1} - z_k},
\end{align*}
which concludes the proof.
\end{proof}

\begin{lemma}
Defining $m_k(x) = \inner{d_k}{x} + c_k + A_k \h(x) + \Lg \w(x)$, we have
\begin{align*}
    A_{k+1} F(y_{k+1}) - m_{k+1}(z_{k+1}) \leq &\, A_k F(y_k) - m_k(z_k) \\
    & +(A_{k+1} - A_k) \inner{\nabla \G(v_k) - g_k}{z_{k+1}} \\
    & - \frac{1}{2} \left( A_k(\muh + 2\Lg) + \Lg (\muw - A_{k+1}) - \Lg\frac{A_k^2}{A_{k+1}}\right) \normsq{z_{k+1} - z_k} \\
    & - \frac{\muh}{2} A_k \left( 1 - \frac{A_k}{A_{k+1}}\right) \normsq{z_{k+1} - y_k}.
\end{align*}
\end{lemma}
\begin{proof}
Summing the inequalities in the above two lemmas above we have that for any $g_\h(z_k) \in \partial \h(z_k)$, 
\begin{align*}
    A_{k+1} F(y_{k+1}) - c_{k+1} - A_{k+1} \h(z_{k+1}) \leq&\, A_{k}F(y_k) - c_k - A_k \h(z_k)\\
    &+ (A_{k+1} - A_k) \inner{\nabla \G(v_k)}{z_{k+1}} - A_k \inner{g_\h(z_k)}{z_{k+1} - z_k}\\
    & - \frac{1}{2} \left( A_k(\muh + 2\Lg) - \Lg A_{k+1} - \Lg\frac{A_k^2}{A_{k+1}}\right) \normsq{z_{k+1} - z_k} \\
    & - \frac{\muh}{2} A_k \left( 1 - \frac{A_k}{A_{k+1}}\right) \normsq{z_{k+1} - y_k}.
\end{align*}
 Substracting $\inner{d_{k+1}}{z_{k+1}}$ on both sides and adding/substracting $\inner{d_k}{z_{k+1}}$ and $\inner{d_k}{z_k}$ on the right-hand side, we get
\begin{align*}
    A_{k+1} F(y_{k+1}) - c_{k+1} - A_{k+1} \h(z_{k+1}) - &\inner{d_{k+1}}{z_{k+1}} \\
     \leq &\, A_{k}F(y_k) - c_k - A_k \h(z_k) - \inner{d_k}{z_k}\\
    &+ \inner{d_k}{z_k} - \inner{d_{k+1}}{z_{k+1}} - \inner{d_k}{z_{k+1}} + \inner{d_k}{z_{k+1}}\\
    &+ (A_{k+1} - A_k) \inner{\nabla \G(v_k)}{z_{k+1}} - A_k \inner{g_\h(z_k)}{z_{k+1} - z_k}\\
    & - \frac{1}{2} \left( A_k(\muh + 2\Lg) - \Lg A_{k+1} - \Lg\frac{A_k^2}{A_{k+1}}\right) \normsq{z_{k+1} - z_k} \\
    & - \frac{\muh}{2} A_k \left( 1 - \frac{A_k}{A_{k+1}}\right) \normsq{z_{k+1} - y_k}\\
    =&\, A_{k}F(y_k) - c_k - A_k \h(z_k) - \inner{d_k}{z_k}\\
    &+ \inner{d_k - d_{k+1}}{z_{k+1}} \\
    &+ (A_{k+1} - A_k) \inner{\nabla \G(v_k)}{z_{k+1}} -  \inner{A_kg_\h(z_k) + d_k}{z_{k+1} - z_k}\\
    & - \frac{1}{2} \left( A_k(\muh + 2\Lg) - \Lg A_{k+1} - \Lg\frac{A_k^2}{A_{k+1}}\right) \normsq{z_{k+1} - z_k} \\
    & - \frac{\muh}{2} A_k \left( 1 - \frac{A_k}{A_{k+1}}\right) \normsq{z_{k+1} - y_k}.
\end{align*}
Now, by first-order optimality conditions of~\eqref{alg1:minimization-step}, $0 \in d_k + A_k \partial \h(z_k) + \Lg \partial w(z_k)$. Therefore there exist subgradients $g_\h'(z_k) \in \partial \h(z_k)$ and $g_w'(z_k) \in \partial w(z_k)$ such that $d_k + A_k g_\h'(z_k) = -\Lg g_w'(z_k)$.  Since the above inequality is true for any $g_\h(z_k) \in \partial \h(z_k)$, it is in particular true for $g_\h'(z_k)$, and thus we have
\begin{align*}
    A_{k+1} F(y_{k+1}) - c_{k+1} - A_{k+1} \h(z_{k+1}) - &\inner{d_{k+1}}{z_{k+1}}\\ \leq& \, A_{k}F(y_k) - c_k - A_k \h(z_k) - \inner{d_k}{z_k}\\
    &+ \inner{d_k - d_{k+1}}{z_{k+1}} \\
    &+ (A_{k+1} - A_k) \inner{\nabla \G(v_k)}{z_{k+1}} + \Lg  \inner{g_\w'(z_k)}{z_{k+1} - z_k}\\
    & - \frac{1}{2} \left( A_k(\muh + 2\Lg) - \Lg A_{k+1} - \Lg\frac{A_k^2}{A_{k+1}}\right) \normsq{z_{k+1} - z_k} \\
    & - \frac{\muh}{2} A_k \left( 1 - \frac{A_k}{A_{k+1}}\right) \normsq{z_{k+1} - y_k}.
\end{align*}
Since $d_k - d_{k+1} = - (A_{k+1} - A_k) g_k$ we get
\begin{align*}
    A_{k+1} F(y_{k+1}) - c_{k+1} - A_{k+1} \h(z_{k+1}) - \inner{d_{k+1}}{z_{k+1}} \leq&\, A_{k}F(y_k) - c_k - A_k \h(z_k) - \inner{d_k}{z_k}\\
    &+ (A_{k+1} - A_k) \inner{\nabla \G(v_k) - g_k}{z_{k+1}}\\
    &+ \Lg\inner{g_\w'(z_k)}{z_{k+1} - z_k} \\
    & - \frac{1}{2} \left( A_k(\muh + 2\Lg) - \Lg A_{k+1} - \Lg\frac{A_k^2}{A_{k+1}}\right) \normsq{z_{k+1} - z_k} \\
    & - \frac{\muh}{2} A_k \left( 1 - \frac{A_k}{A_{k+1}}\right) \normsq{z_{k+1} - y_k}.
\end{align*}
Finally, by strong convexity of $\w$ we have
\begin{align*}
    \Lg\inner{g_\w'(z_k)}{z_{k+1} - z_k} \leq \Lg \left( w (z_{k+1}) - \w(z_k) - \frac{\muw}{2} \normsq{z_{k+1} - z_k} \right), 
\end{align*}
and thus the previous inequality becomes
\begin{align*}
    A_{k+1} F(y_{k+1}) - m_{k+1}(z_{k+1}) \leq &\, A_k F(y_k) - m_k(z_k) \\
    &+ (A_{k+1} - A_k) \inner{\nabla \G(v_k) - g_k}{z_{k+1}} \\
    & - \frac{1}{2} \left( A_k(\muh + 2\Lg) + \Lg (\muw - A_{k+1}) - \Lg\frac{A_k^2}{A_{k+1}}\right) \normsq{z_{k+1} - z_k} \\
    & - \frac{\muh}{2} A_k \left( 1 - \frac{A_k}{A_{k+1}}\right) \normsq{z_{k+1} - y_k}.
\end{align*}
\end{proof}

Those three lemmas allow us to prove~\cref{prop:Lyap-breg}. \\
\textit{Proof of~\cref{prop:Lyap-breg}}.\\
We can rewrite the previous result as
\begin{align*}
    A_{k+1} F(y_{k+1}) - m_{k+1}(z_{k+1})\leq  &\, A_k F(y_k) - m_k(z_k) \\
    & + (A_{k+1} - A_k) \inner{\nabla \G(v_k) - g_k}{z_{k+1} - z_{k}} \\
    & + (A_{k+1} - A_k) \inner{\nabla \G(v_k) - g_k}{ z_{k}}\\
    & - \frac{1}{2} \left( A_k(\muh + 2\Lg) + \Lg (\muw - A_{k+1}) - \Lg\frac{A_k^2}{A_{k+1}}\right) \normsq{z_{k+1} - z_k} \\
    & - \frac{\muh}{2} A_k \left( 1 - \frac{A_k}{A_{k+1}}\right) \normsq{z_{k+1} - y_k}.
\end{align*}
Taking expectation at iteration $k$ conditioned on the previous iterations, we have $\E_k[(A_{k+1} - A_k) \inner{\nabla \G(v_k) - g_k}{z_k} = 0]$.

Moreover, by Fenchel-Young inequality we have that for any $\rho > 0$,
\begin{align*}
    \inner{\nabla \G(v_k) - g_k}{z_{k+1} - z_{k}} &\leq \frac{1}{2\rho} \normsq{\nabla \G(v_k) - g_k}_* + \frac{\rho}{2} \normsq{z_{k+1} - z_k}.
\end{align*}

Taking expectation and using~\cref{assumption:bounded-variance},
\begin{align*}
    (A_{k+1} - A_k) \E_k[\inner{\nabla \G(v_k) - g_k}{z_{k+1} - z_{k}}] \leq \frac{1}{2}(A_{k+1} - A_k) \frac{\varsq}{\rho} + \frac{1}{2} (A_{k+1} - A_k)\rho  \E_k \normsq{z_{k+1} - z_k}
\end{align*}
Thus we have
\begin{align*}
    \E_k \left[ A_{k+1} F(y_{k+1}) - m_{k+1}(z_{k+1}) \right] & \leq A_k F(y_k) - m_k(z_k) + \frac{1}{2} (A_{k+1} - A_k) \frac{\varsq}{\rho} \\
    & - \frac{1}{2} \left( A_k(\muh + 2\Lg) + \Lg (\muw - A_{k+1}) - \Lg\frac{A_k^2}{A_{k+1}} - \rho(A_{k+1} - A_k)\right) \E_k\normsq{z_{k+1} - z_k} \\
    & - \frac{\muh}{2} A_k \left( 1 - \frac{A_k}{A_{k+1}}\right) \E_k\normsq{z_{k+1} - y_k}.
\end{align*}
Now, observe that since $0 \leq A_k / A_{k+1} \leq 1$, the term in $\E_k \normsq{z_{k+1} - y_k}$ is non-positive. Therefore, to obtain the final result, it suffices to set $A_{k+1}$ so that the term in $\E_k\normsq{z_{k+1} - z_k}$ cancels out. In other words, we require
\begin{align*}
    &A_k(\muh + 2\Lg) + \Lg (\muw - A_{k+1}) - \Lg\frac{A_k^2}{A_{k+1}} - \rho(A_{k+1} - A_k) = 0\\ 
    \iff& A_{k+1}(\Lg + \rho ) - A_k(\muh + 2\Lg + \rho) - \Lg \muw+ \Lg \frac{A_k^2}{A_{k+1}} = 0\\
    \iff & A_{k+1}^2 (\Lg + \rho) - A_{k+1} \left( A_k(\muh + 2\Lg + \rho) + \Lg \muw\right) + \Lg A_k^2 = 0\\
    \Leftarrow \quad & A_{k+1} = \frac{A_k(\muh + 2\Lg + \rho) + \Lg \muw + \sqrt{\left(A_k(\muh + 2\Lg + \rho) + \Lg \muw\right)^2 - 4(\Lg + \rho)\Lg A_k^2}}{2(\Lg + \rho)}\\
    \iff& A_{k+1} = \frac{A_k(\muh + 2\Lg + \rho) + \Lg \muw + \sqrt{\Lg^2 \muw^2 + 2\Lg \muw A_k(\muh + 2\Lg + \rho) + A_k^2 \muh^2 + A_k^2\rho^2 + 2A_k^2 \muh\rho + 4 A_k^2 \muh \Lg}}{2(\Lg + \rho)}\\
    \iff& A_{k+1} = \frac{A_k(\muh + 2\Lg + \rho) + \Lg \muw + \sqrt{(\Lg \muw +  \muh A_k)^2 + 4A_k(\Lg^2 \muw + A_k \muh \Lg) + 2\Lg \muw A_k\rho + A_k^2\rho^2 + 2A_k^2 \muh\rho}}{2(\Lg + \rho)}.
\end{align*}
Setting $\rho = \sqrt{\muh \Lg}$ yields the update for $A_{k+1}$ in~\cref{alg:acc-stoch-bregman} and proves the proposition.

\subsection{Proof of\texorpdfstring{~\cref{thm:composite-result}}{\space}}\label{sec:composite_proof}
\begin{proof}
Unrolling the recursion in~\cref{prop:Lyap-breg} and taking total expectation, we have
\begin{align*}
    \E\left[ A_\N F(y_\N) - m_\N(z_\N)\right] \leq A_0 F(y_0) - m_0 (z_0) + A_\N \frac{1}{2\sqrt{\muh\Lg}} \varsq = -\Lg \w(z_0) + A_\N \frac{1}{2\sqrt{\muh\Lg} } \varsq
\end{align*}
Now,
\begin{align*}
    m_\N(\ystar) = A_\N \h(\ystar) + \Lg \w(\ystar) + \sum_{t=0}^{\N-1} \left( A_{t+1} - A_t \right) \left( \G(v_t) - \inner{\nabla \G(v_t)}{v_t} + \inner{g_t}{\ystar}\right)
\end{align*}
Taking total expectation on the $g_t$ we get 
\begin{align*}
    \E[m_\N(\ystar)] &= A_\N \h(\ystar) + \Lg \w(\ystar) + \sum_{t=0}^{\N-1} \left( A_{t+1} - A_t \right) \E\left[ \G(v_t) + \inner{\nabla \G(v_t)}{\ystar - v_t}\right]\\
    &\leq A_\N \h(\ystar) + \Lg \w(\ystar) + A_\N \G(\ystar) = A_\N F(\ystar) + \Lg \w(\ystar)
\end{align*}
where the inequality is by convexity of $\G$. Moreover, $m_\N(z_\N) \leq m_\N(\ystar)$ by construction and thus
\begin{align*}
    \E[A_\N F(y_\N) - A_\N F(\ystar) - \Lg \w(\ystar)] &\leq \E[A_\N F(y_\N) - m_\N(\ystar)] \\
    &\leq \E[A_\N F(y_\N) - m_\N(z_\N)] \\
    &\leq -\Lg \w(z_0) + \frac{A_\N}{2\sqrt{\muh \Lg}} \varsq
\end{align*}
and thus
\begin{align*}
    \E[F(y_\N) - F(\ystar)] &\leq \frac{\Lg (\w(\ystar) - \w(z_0))}{A_\N} + \frac{\varsq}{2\sqrt{\muh \Lg}}\\
    &= \frac{\Lg D_\w(\ystar, y_0)}{A_\N} + \frac{\varsq}{2\sqrt{\muh \Lg}}
\end{align*}
as $y_0 = z_0$ and $\w(\ystar) - \w(y_0)$ is equal to $D_\w (\ystar, y_0) :=w(\ystar) - w(y_0)-\langle g_w(y_0); \ystar-y_0\rangle$ (with $g_w(y_0)\in\partial w(y_0)$) through the choice $g_w(y_0)=0\in\partial w(y_0)$, which is valid as $y_0$ minimizes $w(\cdot)$. 

Finally, we can bound $A_{k+1}$ as
\begin{align*}
    A_{k+1} &\geq A_k \frac{\muh + 2\Lg + \sqrt{\muh \Lg} + \sqrt{\muh^2 + 4\muh \Lg + \sqrt{\muh \Lg}^2 + 2 \muh \sqrt{\muh \Lg}}}{2(\Lg + \sqrt{\muh \Lg})}\\
    &\geq A_k \frac{\muh + 2\Lg +\sqrt{\muh \Lg} + 2 \sqrt{\muh \Lg}}{2(\Lg + \sqrt{\muh \Lg})} \\
    &\geq A_k\left( 1 + \frac{\sqrt{\muh \Lg}}{2(\Lg + \sqrt{\Lg \muh})}\right)\\
    &= A_k \left( 1 + \frac{\sqrt{\muh}}{2( \sqrt{\Lg} + \sqrt{\muh})}\right), 
\end{align*}
{\bf\color{red} with $A_1=\frac{\nu \sqrt{\beta}  }{\sqrt{\beta} +\sqrt{  \mu }}$, leading to 
\begin{align*}
A_k&\geq \frac{\nu \sqrt{\beta}  }{\sqrt{\beta} +\sqrt{  \mu }} \left( 1 + \frac{\sqrt{\muh}}{2( \sqrt{\Lg} + \sqrt{\muh})}\right)^{k-1} \\
&\geq \frac{\nu \sqrt{\beta}  }{\sqrt{\beta} +\sqrt{  \mu }} \exp\left( \frac{(\N-1)\sqrt{\muh}}{4( \sqrt{\Lg} + \sqrt{\muh})}\right),
\end{align*}
for $k\geq 1$, where we used that $1+\frac{x}{2} \geq e^{\frac{x}{4}}$ for $ 0 \leq x \leq 1$. Therefore,
\begin{align*}
    \frac{1}{A_k}  \leq \color{red}\frac{\sqrt{\beta} +\sqrt{  \mu }}{\nu \sqrt{\beta}  } \exp\left( - \frac{{\color{red}(\N-1)} \sqrt{\muh}}{4 \left( \sqrt{\Lg} + \sqrt{\muh}\right)} \right)
    \end{align*}
}
To conclude,
\begin{align*}
    \E[F(y_\N) - F^*] \leq {\color{red}\frac{\sqrt{\beta} +\sqrt{  \mu }}{\nu \sqrt{\beta}  }}\exp\left( - \frac{{\color{red}(\N-1)} \sqrt{\muh}}{{\color{red} 4} \left( \sqrt{\Lg} + \sqrt{\muh}\right)} \right)\Lg D_\w(\ystar, y_0) + \frac{\varsq}{2\sqrt{\muh \Lg}}.
\end{align*}
\end{proof}
\section{Proofs for Accelerated Frank-Wolfe}
\label{app:AFW-proofs}

\subsection{Proof of\texorpdfstring{~\cref{thm:AFW-result}}{\space}}

\subsubsection{Proof of feasibility}
First we show that $x_{k} \in K$ for all $k \in \mathbb{N}$.
\begin{proof}
We prove this by induction. By assumption $x_0 \in K$. Now suppose $x_k \in K$ for some $k \in \mathbb{N}$. We then have
\begin{align*}
     x_{k+1} &= \frac{\Lg}{A_{k+1} + \Lg} x_0 - \frac{d_{k+1}}{A_{k+1} + \Lg}\\
     &= \frac{\Lg}{A_{k+1} + \Lg}x_0 - \frac{d_k + (A_{k+1} - A_k) g_k}{A_{k+1} + \Lg}\\
     &= \frac{(A_k + \Lg)\left(\frac{\Lg}{A_k +\Lg} x_0 - \frac{d_k}{A_k + \Lg}\right) - (A_{k+1} - A_k) g_k}{A_{k+1} + \Lg}\\
     &= \frac{(A_k + \Lg)}{A_{k+1} + \Lg} x_k + \frac{A_{k+1} - A_k}{A_{k+1} + \Lg} (-g_k)\\
     &= \frac{A_k + \Lg}{A_{k+1} + \Lg} x_k + \left( 1- \frac{A_k + \Lg}{A_{k+1} + \Lg}\right) (-g_k).
\end{align*}
By induction hypothesis, $x_k \in K$. We also have $-g_k \in K$ since
\begin{align*}
    -g_k &= - \frac{1}{m} \sum_{i=1}^m g_{k, i} \\
    &= \frac{1}{m} \sum_{i=1}^{m} \argmax_{u\in K} \inner{u}{-v_k + \alpha \Delta_i}.
\end{align*}
In other words, $-g_k$ is a convex combination of elements of $K$ and is thus in $K$. Therefore $x_{k+1} \in K$ as a convex combination of elements of $K$.
\end{proof}

\subsubsection{Proof of dual gap convergence}
\begin{proof}
With $\h(y) = f^*(y)$, $\G(y) = \support_\alpha(-y)$, $\Lg = \frac{R_K M}{\alpha}$ and $\muh = \frac{1}{L}$, we can apply~\cref{prop:Lyap-breg} to $F = \h + \G$ and get
\begin{align*}
    \E_k[A_{k+1} F(y_{k+1}) - m_{k+1}(z_{k+1})] \leq A_k F(y_k) - m_k(z_k) + (A_{k+1} - A_k) \frac{\varsq}{2\sqrt{\muh \Lg}}
\end{align*}
where $m_k(y) = \inner{d_k}{y} + c_k + A_k \h(y) + \Lg \w(y)$, and $z_{k} = \nabla f(x_k)$. Unrolling the recursion as before and taking total expectation we have
\begin{align*}
    \E[A_\N F(y_\N) - m_\N(z_\N)] \leq - \Lg w(z_0) + A_\N \frac{\varsq}{2\sqrt{\muh \Lg} }.
\end{align*}
Recall that we set $\w(y) =  f^*(y) - \inner{x_0}{y}$. Plugging in the function $\h = f^*$, and recalling that $z_k $ minimizes $m_k(y)$, we get
\begin{align*}
    m_{k}(z_k) &= \inf_y \left\{ \inner{d_k}{y} + c_k + (A_k +  \Lg) f^*(y) - \Lg \inner{x_0}{y}\right\}\\
    &= c_k - \sup_y \left\{ \inner{- d_k + \Lg  x_0}{y} - (A_k + \Lg) f^*(y) \right\}\\
    &= c_k - (A_k + \Lg) f\left(  \frac{-d_k + \Lg  x_0}{A_k + \Lg} \right) \\
    &= c_k - (A_k + \Lg) f(x_k).
\end{align*}
Thus we can conclude that
\begin{align*}
    \E[A_\N F(y_\N) + (A_\N + \Lg) f(x_\N)] &\leq -\Lg w(z_0) + A_\N \frac{\varsq}{2\sqrt{\muh \Lg}} + c_\N,
\end{align*}
and in particular
\begin{align}
    \label{eq:AFW-proof-ineq2}
    \E[A_\N F(y_\N) + A_\N f(x_\N) ] &\leq - \beta f(\xstar)-\Lg w(z_0) + A_\N \frac{\varsq}{2\sqrt{\muh \Lg}} + c_\N.
\end{align}
Now, 
\begin{align*}
    c_\N = \sum_{i=0}^{\N-1} (A_{i+1} - A_i) ( \G(v_i) - \inner{\nabla \G(v_i)}{v_i}).
\end{align*}
For any $v \in \Eset$, $\G(v) = \support_\alpha(-v)$ and thus
\begin{align*}
    \G(v) - \inner{\nabla \G(v)}{v} &= \support_\alpha(-v) + \inner{\nabla \support_\alpha(- v)}{v}.
\end{align*}
From Fenchel-Young, $\inner{\nabla \support_\alpha(- v)}{ - v} = \support_\alpha(-v) + \support_\alpha^*(\nabla \support_\alpha(-v))$, and thus
\begin{align*}
    \G(v) - \inner{\nabla \G(v)}{v} = - \support_\alpha^*(\nabla \support_\alpha(-v)).
\end{align*}
Now, for all $u \in \Eset^*$,
\begin{align*}
    \support_\alpha^*(u) &= \sup_{v\in \Eset}\left\{ \inner{u}{v} - \support_\alpha(v)\right\} \\
    &\geq \sup_{v\in \Eset} \left\{\inner{u}{v} - \support(v) - \alpha s_1(0)\right\}\\
    &= \support^*(u) - \alpha \support_1(0)\\
    &= I_K(u) - \alpha \support_1(0)
\end{align*}
where the inequality is from~\cref{prop:perturbed-opt}. In particular, since  $\nabla \support_\alpha(-v)$ is always feasible, we have $\support_\alpha^*(\nabla \support_\alpha(-v)) \geq - \alpha \support_1(0)$. and thus
\begin{align*}
    c_k \leq \sum_{i=0}^{\N - 1} (A_{i+1} - A_i) \alpha \support_1(0) = A_\N \alpha \support_1(0).
\end{align*}
We can then rewrite~\eqref{eq:AFW-proof-ineq2} as
\begin{align*}
    \E[A_\N f^*(y_\N) + A_\N \support_\alpha(- y_\N) + A_\N f(x_\N)] \leq -\Lg w(z_0) - \Lg f(\xstar) + A_\N \frac{\varsq}{2\sqrt{\muh \Lg}} + A_\N \alpha \support_1(0).
\end{align*}
Now, 
\begin{align*}
    w(z_0) = f^*(z_0) - \inner{x_0}{z_0} =  -f(x_0)
\end{align*}
by Fenchel-Young and since $z_0 = \nabla f(x_0)$. Thus
\begin{align*}
    \E[A_\N f^*(y_\N) + A_\N \support_\alpha(- y_\N) + A_\N f(x_\N)] \leq \Lg (f(x_0) - f(\xstar)) + A_\N \frac{\varsq}{2\sqrt{\muh \Lg}} + A_\N \alpha \support_1(0).
\end{align*}
Finally, from~\cref{prop:perturbed-opt},
\begin{align*}
    \support_\alpha(-y) \geq \support(-y)
\end{align*}
for any $y$. We can then conclude
\begin{align*}
    \E \left[ f^*(y_\N) + \support(-y_\N) +  f(x_\N) \right] \leq \frac{\Lg (f(x_0) - f(\xstar))}{A_\N} + \frac{\varsq}{2 \sqrt{\muh \Lg}} + \alpha \support_1(0)
\end{align*}
Bounding $A_\N$ as in~\cref{thm:composite-result} yields
\begin{align}
    \label{eq:AFW-proof-ineq1}
    \E \left[ f^*(y_\N) + \support(-y_\N) +  f(x_\N) \right] \leq {\color{red}\frac{\sqrt{\beta} +\sqrt{  \mu }}{\mu \sqrt{\beta}  }}\exp\left( - \frac{{\color{red}(k-1)}\sqrt{\muh}}{{\color{red}4 } (\sqrt{\Lg} + \sqrt{\muh})} \right) \Lg (f(x_0) - f(\xstar)) + \frac{\varsq}{2 \sqrt{\muh \Lg}} + \alpha \support_1(0) 
\end{align}
It remains to bound the variance $\varsq$. Using~\cref{assumption:improved_variance} we have
\begin{align*}
    \varsq = \E \normsq{g_k - \nabla G(v_k)}_* \leq \frac{4 R_K^2 \rho_{\|\cdot\|_*}}{m}.
\end{align*}
Plugging this back into equation~\eqref{eq:AFW-proof-ineq1}, and plugging in the values of $\Lg = \frac{R_K M}{\alpha}$ and $\muh = \frac{1}{L}$ gives
\begin{align*}
    \E \left[ f(x_\N) - \D(y_\N) \right] \leq 
    \,&{\color{red}\frac{\sqrt{\frac{R_K M}{\alpha}} +\sqrt{  \frac{1}{L} }}{\frac{1}{L} \sqrt{\frac{R_K M}{\alpha}}  }} \exp\left( - \tfrac{{\color{red}(\N -1) }\sqrt{\tfrac{1}{L}}}{{\color{red}4}\left( \sqrt{\tfrac{R_K M}{\alpha}} + \sqrt{\tfrac{1}{L}} \right)} \right) \tfrac{R_K M}{ \alpha} \left( f(x_0) - f(\xstar)\right) \\&\quad+ \tfrac{2R_K^2 \rho_{\norm{\cdot}_*}}{m} \sqrt{\tfrac{\alpha L}{R_K M}} + \alpha \support_1(0). \\
    =\,& {\color{red} \exp\left( - \tfrac{{\color{red}(\N -1) }\sqrt{\tfrac{1}{L}}}{4\left( \sqrt{\tfrac{R_K M}{\alpha}} + \sqrt{\tfrac{1}{L}} \right)} \right) L \left(\sqrt{\frac{R_K M}{\alpha}} +\sqrt{  \frac{1}{L} } \right) \sqrt{\tfrac{R_K M}{ \alpha}} \left( f(x_0) - f(\xstar)\right)} \\&{\color{red}\quad+ \tfrac{2R_K^2 \rho_{\norm{\cdot}_*}}{m} \sqrt{\tfrac{\alpha L}{R_K M}} + \alpha \support_1(0).}
\end{align*}
Assuming $\frac{R_K M}{\alpha} \geq \frac{1}{L}$ yields the result.
\end{proof}

\subsection{Proof of\texorpdfstring{~\cref{thm:AFW-result-bigO}}{\space}}
\begin{proof}
Setting
\begin{align*}
    \alpha = \min \left\{ \frac{\epsilon}{3 \support_1(0) }, \frac{ M \epsilon^2 m^2}{36 L R_K^3 \rho_{\norm{\cdot}_*}^2} \right\},
\end{align*}
it is easy to verify that
\begin{align*}
    \alpha \support_1(0) \leq \frac{\epsilon}{3},
\end{align*}
and that
\begin{align*}
    \frac{2R_K^2 \rho_{\norm{\cdot}_*}}{m} \sqrt{\frac{\alpha L}{R_K M}} \leq \frac{\epsilon}{3}.
\end{align*}
It remains to compute $\N$ such that the first term in the bound of~\cref{thm:AFW-result} is also smaller than $\epsilon/3$. This gives
\begin{align*}
    &\exp\left( -{\color{red}(\N-1)} \frac{\sqrt{\alpha}}{{\color{red}8} \sqrt{L R_K M}}\right) {\color{red} 2L \frac{R_K M}{\alpha} } \left( f(x_0) - f(\xstar)\right) \leq \frac{\epsilon}{3}\\ \iff &\N \geq {\color{red}1+}\frac{{\color{red}8} \sqrt{L R_K M }}{\sqrt{\alpha}} \log \left(  {\color{red} \frac{6L R_K M \left( f(x_0) - f(\xstar)\right)}{\epsilon \alpha} }\right).
\end{align*}
The $\tilde{O}$-complexity follows directly from plugging the value of $\alpha$ in the bound.
\end{proof}

\subsection{Dependence on the norms}\label{app:norm_dependency}
In this section, we compute the value of $\rho_{\norm{\cdot}_*}$ for different $\ell_p$ norms when the underlying vector space has dimension $d$.

\subsubsection{Euclidean norm}\label{app:norm_dependency_euclidean}
In the case of the Euclidean norm, we have $\norm{\cdot} = \norm{\cdot}_* = \norm{\cdot}_2$ and thus
\begin{align*}
    \E \normsq{\frac{1}{m} \sum_{i=1}^m g_{k, i} - \nabla G(v_k)} =& \, \E \normsq{\frac{1}{m} \sum_{i=1}^m g_{k, i} + \nabla \support_\alpha(-v_k)}\\
    =&\, \frac{1}{m^2}  \sum_{i=1}^m \E\normsq{g_{k, i} + \nabla \support_\alpha(-v_k)} \\ & + \frac{2}{m^2}\sum_{1 \leq i < j \leq m} \E[\inner{g_{k, i} + \nabla \support_\alpha(-v_k)}{g_{k, j} + \nabla \support_\alpha(-v_k)}]\\
    =&\, \frac{1}{m^2} \sum_{i=1}^m \E\normsq{g_{k, i} + \nabla \support_\alpha(-v_k)}\\
    =&\, \frac{1}{m^2} \sum_{i=1}^m \E\normsq{ -\argmax_{u \in K}\inner{u}{-v_k + \alpha \Delta_i} + \nabla \support_\alpha(-v_k)}
\end{align*}
where the second equality simply comes from the properties of the Euclidean norm, and the third one comes from the fact that $g_{k, i} + \nabla \support_\alpha(-v_k)$ and $g_{k, j} + \nabla \support_\alpha(-v_k)$ are zero-mean independent random variables for all $i\not = j$.\\
Finally, $\argmax_{u \in K}\inner{u}{v} \in K$ for any $v$, and similarly $\nabla \support_\alpha(v) = \E\left[ \argmax_{u\in K} \inner{u}{v + \alpha \Delta}\right] \in K$ for all $v$. Therefore we have
\begin{align*}
    \E \normsq{\frac{1}{m} \sum_{i=1}^m g_{k, i} - \nabla G(v_k)} &\leq \frac{1}{m^2} \sum_{i=1}^m \max_{u, v \in K} \normsq{u - v} \\
    &\leq \frac{1}{m} \max_{u, v \in K} \left( \norm{u} + \norm{v}\right)^2\\
    &\leq \frac{1}{m} \max_{u, v \in K} 2 \norm{u}^2 + 2\norm{v}^2 \\
    &= \frac{ 4R_K^2}{m}
\end{align*}
and we see that in this case $\rho_{\|\cdot\|_2} = 1$.

\subsubsection{\texorpdfstring{$\ell_p$}{\space}-norms for \texorpdfstring{$2 \leq p < \infty$}{\space}}\label{app:norm_dependency_p_geq_2}
When $\norm{\cdot}_* = \norm{\cdot}_p$, we have $\norm{\cdot} = \norm{\cdot}_q$ for $\frac{1}{p} + \frac{1}{q} = 1$. Since $q \in (1, 2]$, from~\cite{ball1994sharp} we know that $\frac{1}{2} \normsq{\cdot}_q$ is $(q-1)$-strongly convex with respect to $\norm{\cdot}_q$. Therefore $\frac{1}{2} \normsq{\cdot}_p$ is $\frac{1}{q -1}$-smooth with respect to $\norm{\cdot}_p$ \cite{kakade2009duality}. We have
\begin{align*}
    \frac{1}{q - 1} = \frac{1/q}{1 - 1/q} = \frac{p}{q} = p - 1
\end{align*}
so $\frac{1}{2} \normsq{\cdot}_p$ is $(p-1)$-smooth with respect to $\norm{\cdot}_p$. We now closely follow the proof from~\citep[Lemma 2]{kakade2008complexity}. Let $F = \frac{1}{2} \normsq{\cdot}_p$ and let $Z_i = g_{k, i} - \nabla G(v_k)$ so that $\E \normsq{Z_i}_p \leq 4R_K^2$. Let $S_i = \sum_{j=1}^{i-1} S_j$. By smoothness of $F$ we have
\begin{align*}
    F(S_{i-1} + Z_i) \leq F(S_i) + \inner{\nabla F(S_{i-1})}{Z_i} + \frac{p-1}{2} \normsq{Z_i}_p 
\end{align*}
Taking conditional expectation with respect to $Z_1, \dots, Z_{i-1}$, since $\E[Z_i] = 0$ we have
\begin{align*}
    \E_{i}[F(S_i) \mid Z_1, \dots, Z_{i-1}] &\leq F(S_{i-1}) + \frac{p-1}{2} \E \left[ \normsq{Z_i}_p \mid Z_1, \dots, Z_{i-1}\right]\\
    &\leq F(S_{i-1}) + \frac{p-1}{2} 4R_K^2.
\end{align*}
Thus, $F(S_i) - \frac{i(p-1)}{2}$ is a supermartingale and therefore
\begin{align*}
    \E[F(S_n)] = \E\left[\frac{1}{2} \normsq{\sum_{i=1}^m g_{k, i} - \nabla G(v_k)}_p\right] \leq \frac{m (p-1)4R_K^2}{2}
\end{align*}
which shows that $\rho_{\norm{}_p} = p-1$.

\subsubsection{\texorpdfstring{$\ell_p$}{\space}-norms for \texorpdfstring{$1 \leq p < 2$}{\space}}
Recall that for $\infty \geq q > r \geq 1$,
\begin{align}
    \label{eq:norm_equivalence_ineq}
    \norm{\cdot}_q \leq \norm{\cdot}_r \leq d^{1/r - 1/q} \norm{\cdot}_q.
\end{align}

If the norm of interest is $\norm{\cdot}_* = \norm{\cdot}_p$ for $1 \leq p < 2$, we have
\begin{align*}
    \E \normsq{\frac{1}{m} \sum_{i=1}^m g_{k, i} - \nabla G(v_k)}_* &= \E \normsq{\frac{1}{m} \sum_{i=1}^m g_{k, i} - \nabla G(v_k)}_p\\
    &\leq \left(d^{1/p - 1/2}\right)^2 \E \normsq{\frac{1}{m} \sum_{i=1}^m g_{k, i} - \nabla G(v_k)}_2\\
    &\leq d^{(2/p - 1)} \frac{1}{m} \max_{u, v \in K} \normsq{u -v}_2
\end{align*}
where the second inequality comes from the derivation for the Euclidean norm in~\cref{app:norm_dependency_euclidean}. Using~\eqref{eq:norm_equivalence_ineq} we get
\begin{align*}
    \E \normsq{\frac{1}{m} \sum_{i=1}^m g_{k, i} - \nabla G(v_k)}_* &\leq d^{(2/p - 1)} \frac{1}{m} \max_{u, v \in K} \normsq{u -v}_p\\
    &= d^{(2/p - 1)} \frac{4 R_K^2}{m}
\end{align*}
and thus $\rho_{\norm{\cdot}_p} = d^{(2/p - 1)}$.

\subsubsection{\texorpdfstring{$\ell_\infty$}{\space}-norm}
When $\norm{\cdot}_* = \norm{\cdot}_\infty$, we use inequality~\eqref{eq:norm_equivalence_ineq} to get that for any $1 \leq r < \infty$,
\begin{align*}
    \E \normsq{\frac{1}{m} \sum_{i=1}^m g_{k, i} - \nabla G(v_k)}_* &= \E \normsq{\frac{1}{m} \sum_{i=1}^m g_{k, i} - \nabla G(v_k)}_\infty\\
    &\leq \E \normsq{\frac{1}{m} \sum_{i=1}^m g_{k, i} - \nabla G(v_k)}_r.
\end{align*}
Now, if $r \geq 2$, from~\cref{app:norm_dependency_p_geq_2} we have that
\begin{align*}
    \E \normsq{\frac{1}{m} \sum_{i=1}^m g_{k, i} - \nabla G(v_k)}_r &\leq \frac{ r-1}{m} \max_{u, v \in K} \normsq{u - v}_r \\
    &\leq \frac{(r-1)}{m} d^{2/r} \max_{u, v \in K} \normsq{u - v}_\infty\\
    &= \frac{(r-1)4 R_K^2}{m} d^{2/r}.
\end{align*}
Taking $r = 2 + \log d$ (so that $r \geq 2$), we then have
\begin{align*}
    \E \normsq{\frac{1}{m} \sum_{i=1}^m g_{k, i} - \nabla G(v_k)}_* &\leq \frac{(r-1)4 R_K^2}{m} e^{\tfrac{2}{r} \log(d)}\\
    &= \frac{4 R_K^2}{m} (\log d + 1) e^{\tfrac{2 \log d}{2 + \log d}}\\
    &\leq \frac{4 R_K^2}{m} 2(\log d + 1)
\end{align*}
and thus we have $\rho_{\norm{\cdot}_\infty} = 2 (\log d +1)$.

\subsubsection{General norm}
For a general norm, since we are in a finite-dimensional space, there exist constants $c, C > 0$ such that $c \norm{\cdot}_2 \leq \norm{\cdot}_* \leq C \norm{\cdot}_2$. We then have
\begin{align*}
    \E \normsq{\frac{1}{m} \sum_{i=1}^m g_{k, i} - \nabla G(v_k)}_* &\leq C^2 \E \normsq{\frac{1}{m} \sum_{i=1}^m g_{k, i} - \nabla G(v_k)}_2\\
    &\leq C^2 \frac{1}{m^2} \sum_{i=1}^m \max_{u, v \in K} \normsq{u - v}_2\\
    &\leq \frac{C^2}{c^2} \frac{1}{m^2} \sum_{i=1}^m \max_{u, v \in K} \normsq{u - v}_*\\
    &\leq \frac{C^2}{c^2} \frac{4 R_K^2}{m}
\end{align*}
where the second inequality comes from the derivation in~\cref{app:norm_dependency_euclidean}. 
Thus we $\rho_{\norm{\cdot}_*} \leq \frac{C^2}{c^2}$.
\section{Experimental Details}
\label{app:exp-details}
In the experiments we only consider Euclidean norms so that $\norm{\cdot} = \norm{\cdot}_*$.\\
If the entries of $z$ are independently distributed according to a Gumbel distribution with location 0 and scale 1, the probability density function reads
\begin{align*}
    p(z) = e^{- \left(\sum_{i=1}^dz_i + e^{-z_i}\right)}
\end{align*}
so that $\eta(z) = \sum_{i=1}^d z_i + e^{-z_i}$. Thus we have
\begin{align*}
    \nabla \eta(z) = \begin{pmatrix}
    1 - e^{-z_1}\\
    1 - e^{-z_2}\\
    \vdots\\
    1 - e^{-z_d}
    \end{pmatrix}
\end{align*}
and
\begin{align*}
    \normsq{\nabla \eta(z)} = \sum_{j=1}^d (1 - e^{-z_j})^2
\end{align*}
\cref{alg:AFW} and~\cref{alg:R-PFW} require a bound on the value of $M$. We compute it now.
\begin{align*}
    M^2 &= \E \normsq{\nabla \eta(Z)} \\
    & = \int_{\R^d}  \sum_{j=1}^d(1 - e^{-z_j})^2e^{- \left(\sum_{i=1}^dz_i + e^{-z_i}\right)} dz\\
    &= d \int_{\R^d}  (1 - e^{-z_1})^2e^{- \left(\sum_{i=1}^dz_i + e^{-z_i}\right)} dz\\
    &= d \int_\R(1 - e^{-z_1})^2 e^{-(z_1 + e^{-z_1})} \int_\R e^{-(z_2 + e^{-z_2})} \int_\R \dots \int_\R e^{-(z_d + e^{-z_d})} dz_d \dots dz_2 dz_1
\end{align*}
For any $i \geq 2$,
\begin{align*}
    \int_\R e^{-(z_i + e^{-z_i})} dz_i = 1.
\end{align*}
Moreover, one can check that
\begin{align*}
    \int(1 - e^{-z_1})^2 e^{-(z_1 + e^{-z_1})} dz_1 = e^{-e^{-z_1}} + e^{-2z_1 - e^{-z_1}} + C
\end{align*}
where $C$ is some constant. Taking limits one gets
\begin{align*}
    \int_\R(1 - e^{-z_1})^2 e^{-(z_1 + e^{-z_1})}dz_1 = 1
\end{align*}
and thus
\begin{align*}
    M^2 = d.
\end{align*}
The derivation in the case of a multivariate normal distribution is similar.

\end{document}